\documentclass[a4paper, 12pt]{article}
\usepackage{mathtools, enumitem}
\usepackage{amsmath,amsthm,amssymb, authblk}
\usepackage[hyperref, usenames, svgnames,x11names, dvipsnames]{xcolor}
\usepackage[ocgcolorlinks, colorlinks=true, citecolor=Cerulean, linkcolor=Bittersweet, urlcolor=blue]{hyperref}
\newtheorem{theorem}{\bf Theorem}[section]

\DeclareMathOperator*{\esssup}{ess\,sup}

\newcommand{\norm}[2]{\left\lVert #1\right\rVert_{#2}}

\newcommand{\md}{\partial^\bullet}
\swapnumbers
\newtheorem{defn}[theorem]{Definition}
\newtheorem{notation}[theorem]{Important Notation}
\newtheorem{lem}[theorem]{Lemma}

\newtheorem{remark}[theorem]{Remark}

\newcommand{\cts}{\hookrightarrow}
\newcommand{\compact}{\xhookrightarrow{c}}

\newcommand{\grad}{\nabla}
\newcommand{\sgrad}{\nabla_\Omega}
\newcommand{\sgradt}{\nabla_{\Omega(t)}}

\newcommand{\slap}{\Delta_\Omega}
\newcommand{\weaklyto}{\rightharpoonup}

\begin{document}

\title{A Stefan problem on an evolving surface}

\author{Amal Alphonse and Charles M. Elliott}
\affil{Mathematics Institute\\ University of Warwick\\ Coventry CV4 7AL\\ United Kingdom}

\maketitle
%%%%%%%%% Insert author address here
%\address{$^{1}$Mathematics Institute, University of Warwick, Coventry CV4 7AL, United Kingdom\\
%$^{2}$Mathematics Institute, University of Warwick, Coventry CV4 7AL, United Kingdom}

\begin{abstract}
We formulate a Stefan problem on an evolving hypersurface and study the well-posedness of weak solutions given $L^1$ data. To do this, we first develop function spaces and results to handle equations on evolving surfaces in order to give a natural treatment of the problem. Then we consider the existence of solutions for $L^\infty$ data; this is done by regularisation of the nonlinearity. The regularised problem is solved by a fixed point theorem and then uniform estimates are obtained in order to pass to the limit. By using a duality method we show continuous dependence which allows us to extend the results to $L^1$ data.% by means of an approximation argument.
\end{abstract}

\section{Introduction}
The Stefan problem is the prototypical time-dependent free boundary problem. It arises in various forms in many models in the physical and biological sciences \cite{ElliottOckendon, Friedman2, Meir, RodriguesBook2}. In this paper we present the theory of weak solutions associated with the so-called \emph{enthalpy} approach \cite{ElliottOckendon} to the Stefan problem on an evolving curved hypersurface.

Our interest is in the existence, uniqueness and continuous dependence of weak solutions to the Stefan problem
%With $f$ and $u_0$ given and $e \in \mathcal{E}(u)$, we seek solutions of 
\begin{equation}\label{eq:pde}
\begin{aligned}
\md e(t) - \Delta_{\Omega(t)} u(t) + e(t) \sgradt \cdot \mathbf{w}(t) &= f(t) &&\text{in $\Omega(t)$}\\
e(0) &= e_0&&\text{on $\Omega(0)$}\\
e &\in \mathcal{E}(u)
\end{aligned}
\end{equation}
posed on a moving compact hypersurface $\Omega(t) \subset \mathbb{R}^{n+1}$ evolving with (given) velocity field $\mathbf w$, where the energy $\mathcal{E}\colon \mathbb{R} \to \mathcal{P}(\mathbb{R})$ is defined by
\begin{align*}
\mathcal{E}(r) = \begin{cases}
r &\text{for $r < 0$}\\
[0,1] &\text{for $r =0$}\\
r+1 &\text{for $r > 0$}.
\end{cases}
\end{align*}
{Note that $\mathcal{E}$ is a maximal monotone graph in the sense of Br\'{e}zis \cite{Brezis}.} In \eqref{eq:pde}, $\md e$ means the material derivative of $e$ (which we shall also write as $\dot e$) and $\grad_{\Omega(t)}$ and $\Delta_{\Omega(t)}$ are respectively the surface gradient and Laplace--Beltrami operators on $\Omega(t)$.
The novelty of this work is that the Stefan problem itself is formulated on a moving hypersurface and our chosen method to treat this problem, which we believe is naturally suited to equations on moving domains, requires the use of some new function spaces and results that we shall introduce, building upon the spaces and concepts presented in \cite{AlpEllSti, AlpEllSti2}. %The presence of the term $u(t)\sgradt \cdot \mathbf w(t)$. We will make use of the function spaces presented in to treat this problem in a natural way.
There is, as alluded to above, a rich literature associated to Stefan-type problems \cite{Blanchard, Friedman, Kamen, Oleinik, RodriguesDarcy, RodriguesBook}. We will show that arguments similar to those used in the standard setting are also amenable to our problem on a moving hypersurface, thanks in part to the function spaces we decide to use. {Let us remark that the techniques and functional analysis we develop here can be directly applied to study many other nonlinear PDE problems posed on moving domains.} 

Let us work out a possible pointwise formulation of \eqref{eq:pde}. Start by supposing $\Omega(t) = \Omega_l(t) \cup \Omega_s(t) \cup \Gamma(t)$
where $\Omega_l(t)$ and $\Omega_s(t)$ divide $\Omega(t)$ into a liquid and a solid phase (respectively) with an \emph{a priori} unknown interface $\Gamma(t)$. The quantity of interest is the temperature $u(t)\colon \Omega(t) \to \mathbb{R}$, which we suppose satisfies
\begin{align*}
\begin{cases}
u(t) > 0 &\text{ in }\Omega_l(t)\\
u(t) = 0 &\text{ in }\Gamma(t)\\
u(t) < 0 &\text{ in }\Omega_s(t),
\end{cases}
\end{align*}
and thus $u=0$ is the critical temperature where the change of phase occurs.
Define
\begin{align*}
Q_l = \bigcup_{t \in (0,T)}\Omega_l(t) \times \{ t\},\qquad S = \bigcup_{t \in (0,T)}\Gamma(t) \times \{ t\},
\end{align*}
and $Q_s$ similarly. Given $f$ and $u_0$, we formally elucidate in Remark \ref{lem:weakAndPointwiseSolutionLemma} the relationship between \eqref{eq:pde} and the following model describing the temperature $u$:
\begin{equation}\label{eq:stefanModel}
\begin{aligned}
\md u - \slap u + (u+1)\sgrad \cdot \mathbf w &= f && \text{in $Q_l$}\\
\md u - \slap u + u\sgrad \cdot \mathbf w &= f &&\text{in $Q_s$}\\
-(\sgrad u_l - \sgrad u_s)\cdot \mu &=V &&\text{on $S$}\\
u &= 0&&\text{on $S$}\\
u(0) &= u_0 &&\text{on $\Omega(0)$},
\end{aligned}
\end{equation}
where $u_s$ denotes the trace of the restriction $u|_{\Omega_s}$ to the interface $\Gamma$ (likewise with $u_l$), $V(t)$ is the conormal velocity of $\Gamma(t)$ and $\mu(t)$ is the unit conormal vector pointing into $\Omega_l(t)$ (this vector is tangential to $\Omega(t)$ and normal to $\partial\Omega_l(t)$). %{These equations are equalities in the sense that they hold for almost every $t \in [0,T]$ and almost everywhere in $\Omega(t)$.} %The factor of $u+1$ in the divergence term in the first equation is sensible: this is consistent with conservation of mass.
%
%In Lemma \ref{lem:weakAndPointwiseSolutionLemma} below we elucidate the relationship between \eqref{eq:stefanModel} and \eqref{eq:pde}, but first 

We now introduce some notions of a weak solution, similar to \cite{Kamen}. The function spaces $L^p_X$ below will be made precise in \S \ref{sec:prelims} but for now can be thought of as generalisations of Bochner spaces $L^p(0,T;X_0)$ where now $u \in L^p_X$ implies $u(t) \in X(t)$ for almost all $t$ (for a suitable family $\{X(t)\}_{t \in [0,T]}$).
%By a weak solution, we mean the following.
\begin{defn}[Weak solution]Given $f \in L^1_{L^1}$ and $e_0 \in L^1(\Omega_0)$, a \emph{weak solution} of \eqref{eq:pde} is a pair $(u,e) \in L^1_{L^1} \times L^1_{L^1}$ such that $e \in \mathcal{E}(u)$ and 
%\begin{itemize}
%\item[(i)]$u \in L^1_{L^1}$ and $e \in L^1_{L^1}$
%\item[(i)]% and $e_0 \in \mathcal{E}(u_0)$
%\item[(ii)]
there holds
\begin{align*}
-\int_0^T \int_{\Omega(t)}\dot \eta(t)e(t) - \int_0^T \int_{\Omega(t)}u(t) \slap \eta(t) &= \int_0^T \int_{\Omega(t)}f(t)\eta(t) + \int_{\Omega_0}e_0\eta(0)%\label{eq:pdeWeakFormL1}
\end{align*}
for all $\eta \in W(L^\infty \cap H^2, L^\infty)$ with $\slap \eta \in L^\infty_{L^\infty}$ and $\eta(T) = 0$.
%\end{itemize}
%
%A pair $(u,e) \in L^1_{L^1} \times L^1_{L^1}$ is weak solution of class $L^1$ of \eqref{eq:pde} if
%\begin{equation}\label{eq:pdeWeakFormL1}
%\begin{aligned}
%%u &\in L^2_{H^1}\\
%e &\in \mathcal{E}(u),\\
%e_0 &\in \mathcal{E}(u_0),\\
%-\int_0^T \int_{\Omega(t)}\dot \eta(t)e(t) - \int_0^T \int_{\Omega(t)}u(t) \slap \eta(t) &= \int_0^T \int_{\Omega(t)}f(t)\eta(t) + \int_{\Omega_0}e_0\eta(0).
%\end{aligned}
%\end{equation}
%for all $\eta \in W(L^\infty \cap H^2, L^\infty)$ with $\slap \eta \in L^\infty_{L^\infty}$ and $\eta(T) = 0$.
\end{defn}
%In order to carry out the well-posedness proof, we will use the following stronger notion of a weak solution too.
\begin{defn}[Bounded weak solution]\label{defn:boundedWeakSolution}
Given $f \in L^\infty_{L^\infty}$ and $e_0 \in L^\infty(\Omega_0)$, a \emph{bounded weak solution} of \eqref{eq:pde} is a pair $(u,e) \in L^2_{H^1}\times L^\infty_{L^\infty}$ such that $(u,e)$ is a weak solution of \eqref{eq:pde} satisfying
%A \emph{bounded weak solution} of \eqref{eq:pde} is a pair $(u,e) \in L^2_{H^1}\times L^2_{L^2}\cap L^\infty_{L^\infty}$ with $\dot e \in L^2_{H^{-1}}$ such that
%\begin{itemize}
%%\item[(i)] $u \in L^2_{H^1}$ and $e \in L^2_{L^2}\cap L^\infty_{L^\infty}$ with $\dot e \in L^2_{H^{-1}}$
%\item[(i)]$e \in \mathcal{E}(u)$% and $e_0 \in \mathcal{E}(u_0)$
%\item[(ii)]there holds
\begin{equation}\label{eq:pdeWeakForm}
-\int_0^T \int_{\Omega(t)}\dot \eta(t)e(t) + \int_0^T \int_{\Omega(t)}\sgrad u(t) \sgrad \eta(t) = \int_0^T \int_{\Omega(t)}f(t)\eta(t) + \int_{\Omega_0}e_0\eta(0)
%\int_0^T \langle \dot e(t),\eta(t) \rangle_{\Vmt, \Vt} + \int_0^T \int_{\Omega(t)}\sgrad u(t) \sgrad \eta(t) + \int_0^T \int_{\Omega(t)}e(t)\eta(t)\sgrad \cdot \mathbf w = \int_0^T \int_{\Omega(t)}f(t)\eta(t)
\end{equation}
for all $\eta \in W(H^1, L^2)$ with $\eta(T) = 0$.
%\end{itemize}
%
%A pair $(u,e) \in L^\infty_{L^\infty}\cap L^2_{H^1} \times L^2_{L^2}$ is weak solution of class $L^\infty$ of \eqref{eq:pde} if
%\begin{equation}\label{eq:pdeWeakForm}
%\begin{aligned}
%%u &\in L^2_{H^1}\\
%\dot e &\in L^2_{H^{-1}}\\
%e &\in \mathcal{E}(u),\\
%e_0 &\in \mathcal{E}(u_0),\\
%e(0) &= e_0,\\
%\int_0^T \langle \dot e(t),\eta(t) \rangle_{\Vmt, \Vt} + \int_0^T \int_{\Omega(t)}\sgrad u(t) \sgrad \eta(t) + \int_0^T \int_{\Omega(t)}e(t)\eta(t)\sgrad \cdot \mathbf w &= \int_0^T \int_{\Omega(t)}f(t)\eta(t)\end{aligned}
%\end{equation}
%for all $\eta \in W(H^1, L^2)$ with $\eta(T) = 0$.
\end{defn}
%
%We delay the proof of this lemma until p.~\pageref{proof:weakAndPointwiseSolutionLemma} in \S \ref{sec:preliminaryResults}.
%The link between the bounded weak solution and the classical solution is elucidated by Lemma \ref{lem:weakAndPointwiseSolutionLemma} on page \pageref{lem:weakAndPointwiseSolutionLemma}.
%\subsection{Main results}
We prove the following results.
\begin{theorem}[Existence of bounded weak solutions]\label{thm:existenceLinfty}
If $f \in L^\infty_{L^\infty}$, $e_0 \in L^\infty(\Omega_0)$ and $|\Omega| := \sup_{s \in [0,T]}|\Omega(s)| < \infty$, then there exists a bounded weak solution to \eqref{eq:pde}.
\end{theorem}
\begin{theorem}[Uniqueness and continuous dependence of bounded weak solutions]\label{thm:uniquenessAndCtsDependenceLinfty}
If for $i=1$, $2$, $(u^i, e^i)$ are two bounded weak solutions of \eqref{eq:pde} with data $(f^i, e^i_0) \in L^\infty_{L^\infty}\times L^\infty(\Omega_0)$, then%the continuous dependence result
\[\lVert{e^1(t)- e^2(t)}\rVert_{L^1(\Omega(t))} \leq \int_0^t\lVert{f^1(\tau)-f^2(\tau)}\rVert_{L^1(\Omega(\tau))} + \lVert{e^1_{0}-e^2_{0}}\rVert_{L^1(\Omega_0)}\]
for almost all $t$.
%holds.
\end{theorem}
%Our main result is the following.
\begin{theorem}[Well-posedness of weak solutions]\label{thm:wellPosednessL1}If $f \in L^1_{L^1}$, $e_0 \in L^1(\Omega_0)$ and $|\Omega| := \sup_{s \in [0,T]}|\Omega(s)| < \infty$, then there exists a unique weak solution to \eqref{eq:pde}. Furthermore, if for $i=1$, $2$, $(u^i, e^i) \in L^1_{L^1}\times L^1_{L^1}$ are two weak solutions of \eqref{eq:pde} with data $(f^i, e^i_0) \in L^1_{L^1}\times L^1(\Omega_0)$, then %the continuous dependence result
\[\lVert{e^1- e^2}\rVert_{L^1_{L^1}} \leq C_T\left(\lVert{f^1-f^2}\rVert_{L^1_{L^1}} + \lVert{e^1_{0}-e^2_{0}}\rVert_{L^1(\Omega_0)}\right).\]
%holds.
\end{theorem}
%The proof of this theorem utilises the following results.
Below, we shall use the notation $\cts$ and $\compact$ to denote (respectively) a continuous embedding and a compact embedding. 
\section{Preliminaries}\label{sec:prelims}
\subsection{Abstract evolving function spaces}\label{sec:absSpaces}
In \cite{AlpEllSti}, we generalised some concepts from \cite{vierling} and defined the Hilbert space $L^2_{H}$ given a sufficiently smooth parametrised family of Hilbert spaces $\{H(t)\}_{t \in [0,T]}$. We need a generalisation of this theory to Banach spaces.

For each $t \in [0,T]$, let $X(t)$ be a real Banach space with $X_0 := X(0)$. We informally identify the family $\{X(t)\}_{t \in [0,T]}$ with the symbol $X$. Let there be a linear homeomorphism $\phi_t\colon X_0 \to X(t)$ for each $t \in [0,T]$ (with the inverse $\phi_{-t}\colon X(t) \to X_0$) such that $\phi_0$ is the identity. We assume that there exists a constant $C_X$ independent of $t \in [0,T]$ such that
\begin{equation}\label{eq:assOnEvolvingSpaces}
\begin{aligned}
\norm{\phi_t u}{X(t)} &\leq C_X\norm{u}{X_0} &&\forall u \in X_0\\
\norm{\phi_{-t} u}{X_0} &\leq C_X\norm{u}{X(t)} &&\forall u \in X(t).
\end{aligned}
\end{equation}
We assume for all $u \in X_0$ that the map $t \mapsto \norm{\phi_t u}{X(t)}$ is measurable.
%Now we can define the space $L^p_{X}$.
\begin{defn}
Define the Banach spaces
\begin{align*}
L^p_X &= \{u:[0,T] \to \bigcup_{t \in [0,T]}\!\!\!\! X(t) \times \{t\},\quad t \mapsto (\hat u(t), t)\quad \mid \quad\phi_{-(\cdot)} \hat u(\cdot) \in L^p(0,T;X_0 )\}\quad\text{for $p \in [1,\infty)$}\\
L^\infty_{X} &= \{ u \in L^2_{X} \mid \esssup_{t \in [0,T]} \norm{u(t)}{X(t)} < \infty\}
\end{align*}
%for $p \in [1, \infty)$ and for $p=\infty$, define
%\[L^\infty_{X} = \{ u \in L^2_{X} \mid \esssup_{t \in [0,T]} \norm{u(t)}{X(t)} < \infty\}.\]
endowed with the norm
\begin{equation}
\begin{aligned}\label{eq:ip}
\norm{u}{L^p_X} &=\begin{cases}
\left({\int_0^T \norm{u(t)}{X(t)}^p}\right)^{\frac 1 p} &\text{for $p \in [1,\infty)$}\\
\esssup_{t \in [0,T]}\norm{u(t)}{X(t)} &\text{for $p=\infty$}.
\end{cases}
\end{aligned}
\end{equation}
\end{defn}
Note that we made an abuse of notation after the definition of the first space and identified $u(t) = (\hat u(t), t)$ with $\hat u(t)$.
%In other words, $u \in L^\infty_{X}$ if $u \in L^2_{X}$ and for almost all $t \in [0,T]$, $\norm{u(t)}{X(t)} \leq C$.
That \eqref{eq:ip} defines a norm is easy to see once one checks that the integrals are well-defined (the case $p=\infty$ is easy), which can be shown by a straightforward adaptation of the proof of Theorem 2.8 in \cite{AlpEllSti} for the case when each $X(t)$ is separable ({see the appendix}) and the proof of Lemma 3.5 in \cite{vierling} for the non-separable case. The fact that $L^p_X$ is a Banach space follows from Lemma \ref{lem:pushforwardIso} below.% and the completeness of $L^p(0,T;X_0)$. %We relegate the proof for the case $p \in [1,\infty)$ to the appendix because it is comparable to the proof of Theorem 2.7 in \cite{AlpEllSti}.
\begin{notation}Given a function $u \in L^p_X$, the notation $\tilde u$ will be used to mean the pullback $\tilde u(\cdot) := \phi_{-(\cdot)}u(\cdot) \in L^p(0,T;X_0)$, and vice-versa.
\end{notation}
\begin{lem}\label{lem:pushforwardIso}The spaces $L^p(0,T;X_0)$ and $L^p_X$ are isomorphic via $\phi_{(\cdot)}$ with an equivalence of norms:
\begin{equation*}
\begin{aligned}[2]
\frac{1}{C_X}\norm{u}{L^p_X} &\leq \norm{\phi_{-(\cdot)}u(\cdot)}{L^p(0,T;X_0)} \leq C_X\norm{u}{L^p_X}&&\text{for all $u \in L^p_X$}.
\end{aligned}
\end{equation*}
\end{lem}
\begin{proof}
We show the case $p=\infty$ here; an adaptation of the $p=2$ case done in \cite{AlpEllSti} easily proves the lemma for $p \in [1,\infty)$ ({see the appendix}). Let $u \in L^\infty_{X}$.  Measurability of $\tilde u$ follows since $u \in L^2_{X}$. Now, by definition, we have that for all $t \in [0,T] \backslash N$, $\norm{u(t)}{X(t)} \leq A$ where $N$ is a null set and $A=\norm{u}{L^\infty_X}.$ This means that for all $t \in [0,T]\backslash N$, $C_X^{-1}\norm{\tilde u(t)}{X_0} \leq \norm{u(t)}{X(t)} \leq A$ by the assumption \eqref{eq:assOnEvolvingSpaces}, i.e.,
\[\norm{\tilde u}{L^\infty(0,T;X_0)}=\esssup_{t \in [0,T]}\norm{\tilde u(t)}{X_0} \leq C_X{A} =C_X{\norm{u}{L^\infty_{X}}},\]
so $\tilde u \in L^\infty(0,T;X_0)$. Similarly, we conclude that if $\tilde u \in L^\infty(0,T;X_0)$ then $u \in L^\infty_X$.
\end{proof}
\begin{remark}
The dual operator $\phi_{-t}^*\colon X_0^* \to X^*(t)$ is also a linear homeomorphism with $\norm{\phi_{-t}^*}{} = \norm{\phi_{-t}}{}$ and $(\phi_{-t}^*)^{-1} = \phi_t^*$ \cite[Theorem 4.5-2 and \S 4.5]{kreyszig}, and if $X_0$ is separable, $t \mapsto \norm{\phi_{-t}^*f}{X^*(t)}$ is measurable for $f \in X_0^*$; thus, in the separable setting, the dual operator also satisfies the same boundedness properties as $\phi_t$. This means that the spaces $L^p_{X^*}$ are also well-defined Banach spaces given separable $\{X(t)\}_{t \in [0,T]}$ (the map $\phi_{-(\cdot)}^*$ plays the same role as $\phi_{(\cdot)}$ did for the spaces $L^p_X$).
\end{remark}
The following subspaces will be of use later:
\begin{align*}
C^k_X &= \{\xi \in L^2_X \mid \phi_{-(\cdot)}\xi(\cdot) \in C^k([0,T];X_0)\}\quad\text{for $k \in \{0,1,...\}$}\\
%\tilde{C}^1_V &= \{u \mid u(t) = \sum_{j=1}^m \alpha_j(t)\chi_j^t, \text{ $m \in \mathbb{N}$, $\alpha_j \in AC([0,T])$ and $\alpha_j' \in L^2(0,T)$}\}\\
%\end{align*}\
%{and define}
%\begin{align*}
\mathcal{D}_X &= \{\eta \in L^2_X \mid \phi_{-(\cdot)}\eta(\cdot) \in \mathcal{D}((0,T);X_0)\}.
%\mathcal{D}_X[0,T] &= \{\eta \in L^2_X \mid \phi_{-(\cdot)}\eta(\cdot) \in \mathcal{D}([0,T];X_0)\}.
\end{align*}
\subsubsection{Dual spaces}In this subsection, we assume that $\{X(t)\}_{t \in [0,T]}$ is reflexive. In order to retrieve weakly convergent subsequences from sequences that are bounded in $L^p_X$, we need $L^p_X$ to be reflexive. This leads us to consider a characterisation of the dual spaces. We let $p \in [1,\infty)$ and $(p,q)$ be a conjugate pair in this section.
\begin{theorem}\label{thm:dualSpaceIdentification} The space $(L^p_X)^*$ is isometrically isomorphic to $L^{q}_{X^*}$, and hence we may identify $(L^p_X)^* \equiv L^{q}_{X^*}$ and the duality pairing of $f \in L^{q}_{X^*}$ with $u \in L^p_X$ is given by
\[\langle f, u \rangle_{L^{q}_{X^*},L^p_{X}} = \int_0^T \langle f(t), u(t) \rangle_{X^*(t), X(t)}.\]
\end{theorem}
%Precisely, we want to show that $(L^p_X)^* = L^{q_{X^*}}$ where $(p,q)$ is a conjugate pair. 
To prove this theorem, although we can exploit the fact that the pullback is in a Bochner space, showing that the natural duality map is isometric is not so straightforward because $\phi_{(\cdot)}$ is not assumed to be an isometry. In fact, we have to go back to the foundations and emulate the proof for the dual space identification for Bochner spaces; see \cite[\S IV]{uhl}.
%\begin{lem}For $g \in L^{q}_{X^*}$ and $f \in L^p_X$, the expression 
%\[t \mapsto \langle g(t), f(t) \rangle_{X^*(t), X(t)}\]
%is measurable.
%\end{lem}
%\begin{proof}
%---
%Because $\phi_{(\cdot)}^*g(\cdot)=: \tilde g \in L^{q}(0,T;X_0^*)$ and $\tilde f \in L^p(0,T;X_0)$, there exist simple measurable functions
%$\tilde g_n(t) = \sum \tilde g_{i,n}\chi_{E_i}(t)$ and $\tilde f_n(t) = \sum \tilde f_{i,n}\chi_{E_i}(t)$ such that
%\begin{align*}
%\tilde g_n &\to \tilde g \qquad \text{in $L^{q}(0,T;X_0^*)$}\\
%\tilde f_n &\to \tilde f \qquad \text{in $L^{p}(0,T;X_0)$}
%\end{align*}
%and thus for a subsequence, which we relabel,
%\begin{align*}
%\tilde g_n(t) &\to \tilde g(t) \qquad \text{in $X_0^*$}\\
%\tilde f_n(t) &\to \tilde f(t) \qquad \text{in $X_0$}.
%\end{align*}
%which implies
%\begin{align*}
%g_n(t):=  \phi_{-t}^*\tilde g_n(t) =\sum  g_{i,n}\chi_{E_i}(t) &\to %g(t) \qquad \text{in $X^*(t)$}\\
%f_n(t) :=  \phi_{t} \tilde f_n(t) = \sum  f_{i,n}\chi_{E_i}(t) &\to %f(t) \qquad \text{in $X(t)$}.
%\end{align*}
%A calculation shows that
%\[|\langle g(t), f(t) \rangle_{X^*(t),X(t)} - \langle g_n(t), f_n(t) %\rangle_{X^*(t), X(t)}| \to 0,\]
%so $\langle g(t), f(t) \rangle_{X^*(t),X(t)}$ is measurable as it is %the pointwise limit of the functions
%\[\langle g_n(t), f_n(t) \rangle_{X^*(t), X(t)} = \langle \tilde %g_n(t), \tilde f_n(t) \rangle_{X_0^*, X_0}\]
%which are themselves measurable.
%\end{proof}
\begin{lem}
For every $g \in L^{q}_{X^*}$, the expression
\begin{equation}\label{eq:definitionOflittlel}
l(f) = \int_0^T \langle g(t), f(t) \rangle_{X^*(t), X(t)}\qquad\text{for all $f \in L^p_X$}
\end{equation}
defines a functional $l \in (L^p_X)^*$ such that
%\begin{equation}\label{eq:isoiso1}
$\norm{l}{} = \norm{g}{L^{q}_{X^*}}.$
%\end{equation}
\end{lem}
\begin{proof}
Let $g \in L^{q}_{X^*}$ and define $l\colon L^{p}_X \to \mathbb{R}$ by \eqref{eq:definitionOflittlel};
%\[l(f) = \int_0^T \langle g(t), f(t) \rangle_{X^*(t), X(t)}\qquad\text{for all $f \in L^p_X.$}\]
the integral is well-defined by similar reasoning as before (see Lemma 2.13 in \cite{AlpEllSti}). By H\"older's inequality, we have $|l(f)| \leq \norm{g}{L^{q}_{X^*}}\norm{f}{L^p_X},$
%\begin{align*}
%|l(f)| %&\leq \int_0^T \norm{g(t)}{X^*(t)}\norm{f(t)}{X(t)}\\
%&\leq \norm{g}{L^{q}_{X^*}}\norm{f}{L^p_X},
%\end{align*}
so $l \in (L^p_X)^*$ and $\norm{l}{} \leq \norm{g}{L^{q}_{X^*}}.$ We now show the reverse inequality. First suppose $g$ has the form $g(t) = \sum x_{i,t}^*\chi_{E_i}(t)$ where the $x_{i,t}^* \in X^*(t)$ and the $E_i$ are measurable, pairwise disjoint and partition $[0,T].$ It is clear that $\norm{g(t)}{X^*(t)} = \sum\norm{x_{i,t}^*}{X^*(t)}\chi_{E_i}(t).$
Let 
$h(t) = {\norm{g(t)}{X^*(t)}^{{q}\slash{p}}}\slash{\norm{g}{L^{q}_{X^*}}^{{q}\slash{p}}}$
%
%\[h(t) = \frac{\norm{g(t)}{X^*(t)}^{\frac{q}{p}}}{\norm{g}{L^{q}_{X^*}}^{\frac{q}{p}}}\]
which satisfies $\norm{h}{L^p(0,T)}^p = 1$
%\[ \norm{h}{L^p(0,T)}^p = \int_0^T \frac{\norm{g(t)}{X^*(t)}^{q}}{\norm{g}{L^{q}_{X^*}}^{q}} = 1\]
and $\int_0^T \norm{g(t)}{X^*(t)}h(t) = \norm{g}{L^{q}_{X^*}},$
%\begin{align*}
%\int_0^T \norm{g(t)}{X^*(t)}h(t) %&= \frac{1}{\norm{g}{L^{q}_{X^*}}^{\frac{q}{p}}}\int_0^T {\norm{g(t)}{X^*(t)}^{\frac{q}{p}+1}}{} = \frac{1}{\norm{g}{L^{q}_{X^*}}^{\frac{q}{p}}}\int_0^T {\norm{g(t)}{X^*(t)}^{q}} = \norm{g}{L^{q}_{X^*}}^{q - \frac{q}{p}}
% = \norm{g}{L^{q}_{X^*}},
%\end{align*}
%(where we used $\frac{q}{p} +1 =q$) 
hence for any $\epsilon > 0$ we have
\begin{equation}\label{eq:proofDualIsomorphism3}
\int_0^T \norm{g(t)}{X^*(t)}h(t) \geq \norm{g}{L^{q}_{X^*}} - \frac{\epsilon}{2}.
\end{equation}
Now choose $x_{i,t} \in X(t)$, $\norm{x_{i,t}}{X(t)} = 1$ such that
\begin{equation}\label{eq:proofDualSpaceIsomorphism1}
\norm{x_{i,t}^*}{X^*(t)} - \langle x_{i,t}^*, x_{i,t}\rangle_{X^*(t),X(t)} \leq \frac{\epsilon}{2\norm{h}{L^1(0,T)}}.
\end{equation}
Define $f \in L^p_X$ by $f(t) = \sum x_{i,t}h(t)\chi_{E_i}(t)$ and note that $\norm{f}{L^p_X}^p = \norm{h}{L^p(0,T)}^p.$
%\begin{align*}
%\norm{f}{L^p_X}^p &= \int_0^T \norm{\sum x_{i,t}h(t)\chi_{E_i}(t)}{X(t)}^p= \int_0^T \sum\norm{ x_{i,t}}{X(t)}^ph^p(t)\chi_{E_i}(t) = \sum\int_0^T h^p(t)\chi_{E_i}(t) = \norm{h}{L^p(0,T)}^p.
%\end{align*}
We obtain using \eqref{eq:proofDualSpaceIsomorphism1} and \eqref{eq:proofDualIsomorphism3} that $l(f) \geq \norm{g}{L^{q}_{X^*}} - \epsilon.$
%\begin{align*}
%l(f) %&= \int_0^T \langle g(t), f(t) \rangle_{X^*(t),X(t)} = \int_0^T \left \langle \sum_i x_{i,t}^*\chi_{E_i}(t), \sum_j x_{j,t}h(t)\chi_{E_j}(t) \right\rangle_{X^*(t),X(t)}\\
%&= \int_0^T h(t)\sum_i\langle  x_{i,t}^*, x_{i,t} \rangle_{X^*(t),X(t)}\chi_{E_i}(t)\\
%&\geq \int_0^T h(t) \sum_i\left(\norm{x_{i,t}^*}{X^*(t)} -  \frac{\epsilon}{2\norm{h}{L^1(0,T)}}\right)\chi_{E_i}(t)\tag{by \eqref{eq:proofDualSpaceIsomorphism1}}\\
%&= \int_0^T h(t)\norm{g(t)}{X^*(t)} - \int_0^T \frac{\epsilon}{2\norm{h}{L^1(0,T)}}h(t)\\
%&\geq \norm{g}{L^{q}_{X^*}} - \epsilon.\tag{by \eqref{eq:proofDualIsomorphism3}}
%\end{align*}
This proves that $\norm{l}{} = \norm{g}{L^{q}_{X^*}}$ whenever $g(t) = \sum x_{i,t}^*\chi_{E_i}(t)$ is of the stated form. Now suppose $g \in L^{q}_{X^*}$ is arbitrary. Then there exist $\tilde g_n(t) = \sum \tilde g_{i,n}\chi_{E_i}(t)$ with $\tilde g_{i,n} \in X_0^*$ such that $\tilde g_n \to \tilde g$ in $L^{q}(0,T;X_0^*)$ and so the sequence $g_n(t) := \phi_{-t}^*\tilde g_n(t) = \sum \phi_{-t}^*\tilde g_{i,n}\chi_{E_i}(t)$ satisfies $g_n \to g$ in $L^{q}_{X^*}$. Because the $\phi_{-t}^*\tilde g_{i,n} \in X^*(t)$, we know by our efforts above that $l_n\colon L^p_X\to \mathbb{R}$ defined
$l_n(f) = \int_0^T \langle g_n(t), f(t) \rangle_{X^*(t),X(t)}$ 
has norm $\norm{l_n}{} = \norm{g_n}{L^{q}_{X^*}}$. We also have %(recall that $l(f) = \int_0^T \langle g(t), f(t) \rangle_{X^*(t), X(t)}$)
\[\norm{l_n - l}{} \leq \norm{g_n - g}{L^{q}_{X^*}} \to 0\]
which implies $\lim_{n \to \infty} \norm{l_n}{} = \norm{l}{}$ and also $\lim_{n \to \infty} \norm{l_n}{} = \lim_{n \to \infty} \norm{g_n}{L^{q}_{X^*}} = \norm{g}{L^{q}_{X^*}}.$
\end{proof}
We have shown that $\mathcal{J}\colon L^{q}_{X^*} \to (L^p_X)^*$ defined by $\mathcal J(g) := l(\cdot) = \int_0^T \langle g(t), (\cdot)(t) \rangle_{X^*(t),X(t)}$
is isometric: $\norm{\mathcal Jg}{(L^q_X)^*} = \norm{l}{} = \norm{g}{L^{q}_{X^*}}$. We now show that $\mathcal J$ is onto. Given $l \in (L^p_X)^*$, define $\tilde L\colon L^p(0,T;X_0) \to \mathbb{R}$ by $\tilde L(\tilde v) = l(\phi_{(\cdot)}\tilde v(\cdot)) = l(v)$ for all $\tilde v \in L^p(0,T;X_0)$. It is obvious that $\tilde L \in L^p(0,T;X_0)^*,$ and by the dual space identification for Bochner spaces, there exists an $\tilde L^* \in L^{q}(0,T;X_0^*)$ such that
\begin{align*}
\langle l, v \rangle_{(L^p_X)^*, L^p_X} &= \langle \tilde L, \tilde v \rangle_{L^p(0,T;X_0)^*,L^p(0,T;X_0)} %= \int_0^T \langle \tilde L^*(t), \tilde v(t) \rangle_{X_0^*,X_0}
 = \int_0^T \langle \phi_{-t}^* \tilde L^*(t), v(t) \rangle_{X^*(t),X(t)},
\end{align*}
so $\mathcal J (\phi_{-(\cdot)}^*\tilde L^*(\cdot)) = l$ where $\phi_{-(\cdot)}^*\tilde L^*(\cdot) \in L^{q}_{X^*}.$ Hence $\mathcal J$ is onto, and we have proved Theorem \ref{thm:dualSpaceIdentification}.
\subsection{Function spaces on evolving surfaces}
%\subsubsection*{Assumptions on the evolving hypersurface}
We now make precise the assumptions on the evolving surface $\Omega(t)$ our Stefan problem is posed on and we discuss function spaces in the context of the previous subsections.
For each $t \in [0,T],$ let $\Omega(t) \subset \mathbb{R}^{n+1}$ be an orientable compact (i.e., no boundary) $n$-dimensional hypersurface of class $C^3$, and assume the existence of a flow $\Phi\colon [0,T] \times \mathbb{R}^{n+1} \to \mathbb{R}^{n+1}$ such that for all $t \in [0,T]$, with $\Omega_0 := \Omega(0)$, the map $\Phi_t^0(\cdot):=\Phi(t,\cdot)\colon \Omega_0 \to \Omega(t)$ is a $C^3$-diffeomorphism that satisfies
%\begin{equation}\label{paramvel}
%\begin{aligned}
$\frac{d}{dt}\Phi^0_t(\cdot) = \mathbf w(t,\Phi^0_t(\cdot))$ and
$\Phi^0_0(\cdot) = \text{Id}(\cdot)$
%\end{aligned}
%\end{equation}
%So given a point $x_0 \in \Omega_0$, $\Phi^0_t(x_0) \in \Omega(t)$.
for a given $C^2$ velocity field $\mathbf w\colon [0,T]\times\mathbb{R}^{n+1} \to \mathbb{R}^{n+1}$, which we assume satisfies the uniform bound
$|\sgradt \cdot \mathbf w(t)| \leq C$ for all $t \in [0,T]$. A $C^2$ normal vector field on the hypersurfaces is denoted by $\mathbf \nu\colon [0,T]\times \mathbb{R}^{n+1} \to \mathbb{R}^{n+1}$. It follows that the Jacobian $J^0_t := \det \mathbf{D}\Phi^0_t$ is $C^2$ and is uniformly bounded away from zero and infinity.

%\subsubsection*{Function spaces on evolving surfaces}
For $u\colon \Omega_0 \to \mathbb{R}$ and $v\colon \Omega(t) \to \mathbb{R}$, define the pushforward $\phi_t u = u \circ \Phi^t_0$ and pullback $\phi_{-t}v = v \circ \Phi_t^0$, where $\Phi^t_0 := (\Phi_t^0)^{-1}$. %Note that $\phi_{-t}$ is the inverse map of $\phi_t$. 
We showed in \cite{AlpEllSti2} that 
%\begin{align*}
$\phi_{t}\colon L^2(\Omega_0) \to L^2(\Omega(t))$ and $\phi_{t}\colon H^1(\Omega_0) \to H^1(\Omega(t))$
%\end{align*}
are linear homeomorphisms (with uniform bounds) and (thus) with $L^2 \equiv \{L^2(\Omega(t))\}_{t \in [0,T]},$ $H^1 \equiv \{H^1(\Omega(t))\}_{t \in [0,T]}$ and $H^{-1} \equiv \{H^{-1}(\Omega(t))\}_{t \in [0,T]},$ the spaces $L^2_{L^2}$, $L^2_{H^1}$ and $L^2_{H^{-1}}$ are well-defined (see \cite{AlpEllSti2, DziukElliottL2} for an overview of Lebesgue and Sobolev spaces on hypersurfaces) and we let $L^2_{H^1} \subset L^2_{L^2} \subset L^2_{H^{-1}}$ be a Gelfand triple.

A function $u \in C^1_{L^2}$ has a \emph{strong material derivative} defined by $\dot u(t) = \phi_t\left(\frac{d}{dt}(\phi_{-t}u(t))\right).$ Given a function $u \in L^2_{H^1}$, we say that it has a \emph{weak material derivative} $g \in L^2_{H^{-1}}$ if
\[ (u, \dot \eta)_{L^2_{L^2}} = -\langle g, \eta \rangle_{L^2_{H^{-1}}, L^2_{H^1}} - (u, \eta\sgrad \cdot \mathbf w)_{L^2_{L^2}} \qquad \forall \eta \in \mathcal{D}_{H^1}\]
holds, and we write $\dot u$ or $\md u$ instead of $g$.
Define the Hilbert spaces (see \cite{AlpEllSti, AlpEllSti2} for more details)
\begin{align*}
\mathcal W(H^1(\Omega_0), H^{-1}(\Omega_0)) &= \{ u \in L^2(0,T;H^1(\Omega_0)) \mid u' \in L^2(0,T;H^{-1}(\Omega_0))\}\\
W(H^1, H^{-1}) &= \{ u \in L^2_{H^1} \mid \dot u \in L^2_{H^{-1}}\}
\end{align*}
endowed with the natural inner products. For subspaces $X \cts H^1$ and $Y \cts H^{-1}$, we also define the subset $W(X,Y) \subset W(H^1, H^{-1})$ in the natural manner.
\begin{lem}[See \cite{AlpEllSti, AlpEllSti2}]Let either $X=W(H^1, H^{-1})$ and $X_0 = \mathcal{W}(H^1(\Omega_0), H^{-1}(\Omega_0))$, or $X=W(H^1, L^2)$ and $X_0=\mathcal{W}(H^1(\Omega_0), L^2(\Omega_0))$. For such pairs, the space $X$ is isomorphic to $X_0$ via $\phi_{-(\cdot)}$ with an equivalence of norms:
\[C_1\norm{\phi_{-(\cdot)}v(\cdot)}{X_0} \leq \norm{v}{X} \leq C_2\norm{\phi_{-(\cdot)}v(\cdot)}{X_0}.\]
\end{lem}
We showed in \cite{AlpEllSti, AlpEllSti2} that for $u$, $v \in W(H^1, H^{-1})$, the map $t \mapsto (u(t), v(t))_{L^2(\Omega(t))}$ is absolutely continuous, and 
\[\frac{d}{dt}\int_{\Omega(t)}u(t)v(t) = \langle \dot u(t), v(t) \rangle + \langle \dot v(t), u(t)\rangle + \int_{\Omega(t)}u(t)v(t)\sgrad \cdot \mathbf w(t)\]
holds for almost all $t$, where the duality pairing is between $H^{-1}(\Omega(t))$ and  $H^1(\Omega(t))$.
%The first space is a Hilbert space due to \cite{AlpEllSti} and \cite{AllEllSti2}.
%\begin{cor}The space $L^1_{L^1}$ is well-defined and the map $\phi_{(\cdot)}$ is an isomorphism between $L^1(0,T;L^1(\Omega_0))$ and $L^1_{L^1}$, and there is an equivalence of norms.
%\end{cor}
%
%\begin{defn}[The space $L^\infty_{L^\infty}$]
%Define the space 
%\[L^\infty_{L^\infty} = \{ u \in L^2_{L^2} \mid \esssup_{t \in [0,T]}\esssup_{x \in \Omega(t)}{|u(t,x)|} < \infty\}.\]
%In other words, $u \in L^\infty_{L^\infty}$ if $u \in L^2_{L^2}$ and for almost all $t \in %[0,T]$, $|u(t,x)| \leq C$ for almost all $x \in \Omega(t)$.
%
%Give it the norm
%\[\norm{u}{L^\infty_{L^\infty}} = \esssup_{t \in [0,T]}\esssup_{x \in \Omega(t)}{|u(t,x)|}.\]
%\end{defn}
\subsubsection{Some useful results}
In this subsection, $p$ and $q$ are not necessarily conjugate. The first part of the following lemma is a particular realisation of Lemma \ref{lem:pushforwardIso}. 
\begin{lem}\label{lem:lplqisomorphism}For $p$, $q \in [1, \infty]$, the spaces $L^p_{L^q}$ and $L^p(0,T;L^q(\Omega_0))$ are isomorphic via the map $\phi_{(\cdot)}$ with an equivalence of norms. If $q=\infty$ the spaces are isometrically isomorphic. The embedding $L^\infty_{L^\infty} \subset L^p_{L^q}$ is continuous.
\end{lem}
\begin{lem}\label{lem:compactEmbedding}The space $W(H^1, H^{-1})$ is compactly embedded in $L^2_{L^2}$.
\end{lem}
%\begin{lem}\label{lem:LinftyL2continuousEmbedding}For any $p, q \geq 1$, the embedding $L^\infty_{L^\infty} \subset L^p_{L^q}$ is continuous.
%\end{lem}
\begin{theorem}[Dominated convergence theorem for $L^p_{L^q}$]\label{thm:DCT}
Let $p$, $q \in [1,\infty)$. Let $\{w_n\}$ and $w$ be functions such that $\{\tilde w_n\}$ and $\tilde w$ are measurable (eg. membership of $L^1_{L^1}$ will suffice). If for almost all $t \in [0,T]$, 
\begin{equation*}
\begin{aligned}
w_n(t) &\to w(t) &&\text{almost everywhere in $\Omega(t)$}\\
%f(w_n(t)) &\to f(w(t)) &&\text{almost everywhere in $\Omega(t)$}\\
\exists g \in L^p_{L^q} : |w_n(t)| &\leq g(t)&&\text{almost everywhere in $\Omega(t)$ and for all $n$},
\end{aligned}
\end{equation*}
then $w_n \to w$ in $L^p_{L^q}$.
\end{theorem}
%\begin{proof}
%\end{proof}
%\begin{remark}
%In the above proof we used that $\Phi^t_0$ maps null sets to null sets. %To see this, consider
%\[0=\int_{M_t}\mathrm{d}\mu(t) = \int_{\Phi^t_0(M_t)}J_t^0 \;\mathrm{d}\mu_0\]
%and note that this implies $\mu(0)(\Phi^t_0(M_t)) = 0$ since $J_t^0$ is positive.
%\end{remark}
%We need the following result for an $L^\infty$ estimate.
\begin{lem}\label{lem:ibpWithMaxFunction}If $u \in W(H^1, H^{-1})$, then
\begin{equation}\label{eq:ibpWithMaxFunction}
2\int_0^T \langle \dot u(t), u^+(t) \rangle_{H^{-1}(\Omega(t)), H^1(\Omega(t))} = \int_{\Omega(T)}u^+(T)^2 - \int_{\Omega_0}u^+(0)^2 - \int_0^T \int_{\Omega(t)}u^+(t)^2\sgrad \cdot \mathbf w.
\end{equation}
\end{lem}
\begin{proof}
By density, we can find $\{u_n\} \subset W(H^1, L^2)$ with $u_n \to u$ in $W(H^1, H^{-1})$. It follows that $\md(u_n^+) = \dot u_n\chi_{u_n \geq 0} \in L^2_{L^2}$ (this is sensible because $w \in H^1(\Omega)$ implies $w^+ \in H^1(\Omega)$) and therefore
%\begin{align*}
%\int_0^T \langle \dot u_n(t), u_n^+(t) \rangle_{H^{-1}(\Omega(t)), H^1(\Omega(t))} &= \int_0^T \int_{\Omega(t)}\md u_n^+(t)u_n^+(t) \\
%&= \int_0^T \frac{1}{2}\frac{d}{dt}\int_{\Omega(t)}u_n^+(t)^2 - \frac{1}{2}\int_0^T \int_{\Omega(t)}u_n^+(t)^2\sgrad \cdot \mathbf w\\
%&= \frac{1}{2}\int_{\Omega(T)}u_n^+(T)^2-\frac{1}{2}\int_{\Omega_0}u_n^+(0)^2 - \frac{1}{2}\int_0^T \int_{\Omega(t)}u_n^+(t)^2\sgrad \cdot \mathbf w,
%\end{align*}
%i.e., 
\eqref{eq:ibpWithMaxFunction} holds for $u_n$. Since $W(H^1, H^{-1}) \hookrightarrow C^0_{L^2}$, it follows that %$u_n(t) \to u(t)$ in $L^2(\Omega(t))$ for each $t$. This further implies that 
$u_n^+(t) \to u^+(t)$ in $L^2(\Omega(t))$ (for example see \cite[Lemma 2.88]{nonsmooth} or \cite[Lemma 1.22]{nonlinearpotential}). So we can pass to the limit in the first two terms on the right hand side. 

Now we just need to show that $u_n^+ \to u^+$ in $L^2_{H^1}$. %We have 
%\begin{align*}
%\norm{u_n^+-u^+}{L^2_{H^1}}^2 &= \int_0^T \int_{\Omega(t)}|u_n^+(t,x)-u^+(t,x)|^2 + |\sgrad u_n^+(t,x) - \sgrad u^+(t,x)|^2\\
%&\leq \int_0^T \int_{\Omega(t)}|u_n(t,x)-u(t,x)|^2 + |\sgrad u_n^+(t,x) - \sgrad u^+(t,x)|^2
%\end{align*}
It is easy to show the convergence in $L^2_{L^2}$, so we need only to check the convergence of the gradient. Let $g(r) = \chi_{\{r > 0\}}$. Then, using $g \leq 1$,
\begin{align*}
|\sgrad u_n^+(t,x) - \sgrad u^+(t,x)| %&= |g(u_n)\sgrad u_n(t,x) - g(u)\sgrad u(t,x)|\\
%&\leq |g(u_n)||\sgrad u_n(t,x)-\sgrad u(t,x)| + |g(u_n)-g(u)||\sgrad u(t,x)|\\
&\leq |\sgrad u_n(t,x)-\sgrad u(t,x)| + |g(u_n(t,x))-g(u(t,x))||\sgrad u(t,x)|.
\end{align*}
%and the integral of the first term is easy to handle. 
For the second term, let us note that since $u_n \to u$ in $L^2_{H^1}$, for almost all $t$, $u_n(t,x) \to u(t,x)$ almost everywhere in $\Omega(t)$ for a subsequence (which we have not relabelled). Let us fix $t$. Then for almost every $x \in \Omega(t)$, it follows that $g(u_n(t,x))\sgrad u(t,x) \to g(u(t,x))\sgrad u(t,x)$ pointwise. Because $g \leq 1$, the dominated convergence theorem gives overall $\sgrad u_n^+ \to \sgrad u^+$ in $L^2_{L^2}$.
\end{proof}
\subsection{Preliminary results}\label{sec:preliminaryResults}
\begin{remark}\label{lem:weakAndPointwiseSolutionLemma}It is well-known in the standard setting that a mushy region (the interior of the set where the temperature is zero) {can arise in the presence of heat sources \cite{Bertsch, ElliottOckendon}}; with no heat sources, the initial data may give rise to mushy regions. We will content ourselves with the following heuristic calculations under the assumption that there is no mushy region. 

Let the bounded weak solution of \eqref{eq:pde} (in the sense of Definition \ref{defn:boundedWeakSolution}) have the additional regularity $u \in W(H^1, L^2)$ and $\slap u \in L^2_{L^2},$ and suppose that the sets $\Omega_l(t) = \{ u > 0\}$ and $\Omega_s(t) = \{ u < 0\}$ divide $\Omega(t)$ with a common interface $\Gamma(t)$, which we assume is a sufficiently smooth $n-$dimensional hypersurface (of measure zero with respect to the surface measure on $\Omega(t)$). %\begin{align*}
%u &\in W(H^1, L^2),\\
%\slap u &\in L^2_{L^2},
%\end{align*}
Then the bounded weak solution is also a classical solution in the sense of \eqref{eq:stefanModel}. To see this, suppose that $(u,e)$ is a weak solution satisfying the equality in \eqref{eq:pdeWeakForm}. The integration by parts formula on each subdomain of $\Omega$ implies
%\begin{align*}
%\int_0^T \int_{\Omega_s(t)}\sgrad u(t) \sgrad \eta(t) &= - \int_0^T \int_{\Omega_s(t)}\eta (t)\slap u(t) + \int_0^T \int_{\Gamma(t)}\eta(t) \sgrad u_s(t) \cdot \mu \\
%\int_0^T \int_{\Omega_l(t)}\sgrad u(t) \sgrad \eta(t) &= - \int_0^T \int_{\Omega_l(t)}\eta (t)\slap u(t) - \int_0^T %\int_{\Gamma(t)}\eta(t) \sgrad u_l(t) \cdot \mu,
%\end{align*}
%which we add together to get
\begin{align}
\int_0^T \int_{\Omega(t)}\sgrad u(t) \sgrad \eta(t) &= - \int_0^T \int_{\Omega(t)}\eta (t)\slap u(t) + \int_0^T \int_{\Gamma(t)}\eta(t) (\sgrad u_s(t) - \sgrad u_l(t)) \cdot \mu.\label{eq:15}
\end{align}
With $e(t)\eta(t)\sgrad \cdot \mathbf w =\sgrad \cdot(e(t)\eta(t)\mathbf w) - \mathbf w \cdot \sgrad (e(t)\eta(t))$ and the divergence theorem \cite[\S 2.2]{DziukElliottL2},
\begin{align*}
\int_0^T \int_{\Omega_s(t)}e(t)\eta(t)\sgrad \cdot \mathbf w %&=\int_0^T \int_{\Omega_s(t)}\sgrad \cdot(e(t)\eta(t)\mathbf w) - \int_0^T \int_{\Omega_s(t)}\mathbf w \cdot \sgrad (e(t)\eta(t))\\
&=\int_0^T \int_{\Gamma(t)}e(t)\eta(t)\mathbf w\cdot \mu + \int_0^T \int_{\Omega_s(t)}\mathbf w \cdot (e(t)\eta(t)\nu H - \sgrad (e(t)\eta(t))).
\end{align*}
We use this result in the formula for integration by parts over time over $\Omega_s$:
\begin{align*}
\int_0^T \int_{\Omega_s(t)}\dot \eta(t)e(t)&= \int_0^T \frac{d}{dt}\int_{\Omega_s(t)}e(t)\eta(t) - \int_0^T \int_{\Omega_s(t)}\dot e(t)\eta(t) -\int_0^T \int_{\Gamma(t)}e(t)\eta(t)\mathbf w\cdot \mu\\
&\quad - \int_0^T \int_{\Omega_s(t)}e(t)\eta(t)\mathbf w \cdot \nu H + \int_0^T \int_{\Omega_s(t)}\mathbf w \cdot \sgrad (e(t)\eta(t)).
\end{align*}
A similar expression over $\Omega_l$ can also be derived this way, the difference being that the term with $\mu$ has the opposite sign. Then, using $\dot e =\md (\mathcal{E}(u))  = \dot u,$ $e_s(t)|_{\Gamma(t)} = 0$, and $e_l(t)|_{\Gamma(t)} = 1,$ we get
\begin{align}
\nonumber \int_0^T \int_{\Omega(t)}\dot \eta(t)e(t) %&= \int_0^T \int_{\Omega_s(t)}\dot \eta(t)e(t)+ \int_0^T \int_{\Omega_l(t)}\dot \eta(t)e(t)\\
%\nonumber &= \int_0^T \frac{d}{dt}\int_{\Omega(t)}e(t)\eta(t) - \int_0^T \int_{\Omega(t)}\dot u(t)\eta(t) -\int_0^T \int_{\Gamma(t)}(e_s(t)-e_l(t))\eta(t)\mathbf w\cdot \mu\\
%\nonumber &\quad - \int_0^T \int_{\Omega(t)}e(t)\eta(t)\mathbf w \cdot \nu H + \int_0^T \int_{\Omega(t)}\mathbf w \cdot \sgrad (e(t)\eta(t))\\
\nonumber &= \int_0^T \frac{d}{dt}\int_{\Omega(t)}e(t)\eta(t) - \int_0^T \int_{\Omega(t)}\dot u(t)\eta(t) + \int_0^T \int_{\Gamma(t)}\eta(t)\mathbf w\cdot \mu\\
&\quad - \int_0^T \int_{\Omega(t)}e(t)\eta(t)\mathbf w \cdot \nu H + \int_0^T \int_{\Omega(t)}\mathbf w \cdot \sgrad (e(t)\eta(t))\label{eq:14}.
\end{align}
Since by the partial integration formula $\int_{\Omega(t)}\underline{D}_i(g) = \int_{\Omega(t)}gH\nu_i$, we have (with $g=\mathbf w_ie(t)\eta(t)$) that the fourth term in the right hand side of \eqref{eq:14} is
\begin{align*}
\int_{\Omega(t)}e(t)\eta(t)\mathbf w \cdot \nu H = \sum_i\int_{\Omega(t)}e(t)\eta(t)\mathbf w_i \nu_i H %= \sum_i\int_{\Omega(t)}\underline{D}_i(\mathbf w_i e(t)\eta(t)) = \int_{\Omega(t)}\sgrad \cdot (e(t)\eta(t)\mathbf w)\\
= \int_{\Omega(t)}\sgrad (e(t)\eta(t))\cdot \mathbf w+ \int_{\Omega(t)}\eta(t)e(t)\sgrad \cdot \mathbf w.
\end{align*}
So the calculation \eqref{eq:14} becomes
\begin{align}\label{eq:amal}
\int_0^T \int_{\Omega(t)}\dot \eta(t)e(t) &= \int_0^T \left(\frac{d}{dt}\int_{\Omega(t)}e(t)\eta(t) - \int_{\Omega(t)}(\dot u(t)\eta(t)+\eta(t)e(t)\sgrad \cdot \mathbf w) +  \int_{\Gamma(t)}\eta(t)\mathbf w\cdot \mu\right).
\end{align}
Now, taking the weak formulation \eqref{eq:pdeWeakForm} and substituting \eqref{eq:amal} together with the expression for the spatial term \eqref{eq:15}, we get for $\eta$ with $\eta(T)=\eta(0) = 0$
\begin{align*}
&\int_0^T \int_{\Omega(t)}f(t)\eta(t) = -\int_0^T \int_{\Omega(t)}\dot \eta(t)e(t) + \int_0^T \int_{\Omega(t)}\sgrad u(t) \sgrad \eta(t)\\
&= \int_0^T\int_{\Omega(t)}(\dot u(t) + e(t)\sgrad \cdot \mathbf w -\slap u(t))\eta (t) + \int_{\Gamma(t)}\eta(t) \left((\sgrad u_s(t) - \sgrad u_l(t)) \cdot \mu - (\mathbf w \cdot \mu)\right).
\end{align*}
Taking $\eta$ to be compactly supported in $Q_s$, and afterwards taking $\eta$ compactly supported in $Q_l$, we recover exactly the first two equations in \eqref{eq:stefanModel}.
%\begin{align*}
%\dot u(t) -\slap u(t) + u(t)\sgrad \cdot \mathbf w &= f(t) \quad \text{a.e. in $Q_s$}\\
%\dot u(t) -\slap u(t) + (u(t)+1)\sgrad \cdot \mathbf w &= f(t) \quad \text{a.e. in $Q_l$}
%\end{align*}
So we may drop the first integral on the left and the right hand side. Then with a careful choice of $\eta$, we will obtain
%\[(\sgrad u_s(t) - \sgrad u_l(t)) \cdot \mu - (\mathbf w \cdot \mu)=0\]
%which is 
precisely the interface condition in \eqref{eq:stefanModel}.
%The next lemma is vital for the continuous dependence of (bounded) weak solutions.
\end{remark}
\begin{lem}\label{lem:auxiliaryPDE}
Given $\xi \in C^1(\Omega_0)$ and $\tilde \alpha \in C^2([0,T]\times \Omega_0)$ satisfying $0 < \epsilon \leq \alpha \leq \alpha_0$ a.e., there exists a unique solution $\varphi \in W(H^1, L^2)$ with $\slap \varphi \in L^2_{L^2}$ to
\begin{align}
\dot \varphi - \alpha(x,t)\slap \varphi &= 0\label{eq:auxiliaryPDE}\\
\nonumber \varphi(x,0)&= \xi(x)
\end{align}
satisfying $\norm{\varphi}{L^\infty_{L^\infty}} \leq \norm{\xi}{L^\infty(\Omega_0)}$ and (cf. \cite[Chapter V, \S 9]{lady})
\begin{align}
\int_0^t \int_{\Omega(\tau)}(\dot \varphi(\tau))^2+ \int_0^t \int_{\Omega(\tau)}\alpha|\slap \varphi|^2 +\int_{\Omega(t)}|\sgrad \varphi(t)|^2&\leq (1+\alpha_0)(1+{e^{2C_{\mathbf w}(1+\alpha_0)t}})\int_{\Omega_0}|\sgrad \xi|^2\label{eq:auxiliaryPDEInequality}.
\end{align}
%and if $\alpha_0 = 1+\epsilon$,
%\begin{align*}
%\int_0^t \int_{\Omega(\tau)}(\dot \varphi(\tau))^2+ \int_0^t \int_{\Omega(\tau)}\alpha|\slap \varphi|^2 +\int_{\Omega(t)}|\sgrad \varphi(t)|^2&\leq (2+\epsilon)(1+{e^{2C_w(2+\epsilon)t}})\int_{\Omega_0}|\sgrad \xi|^2.
%\end{align*}
\end{lem}
\begin{proof}%For the well-posedness,
%\paragraph{Well-posedness}
%The weak formulation is
%\begin{align*}
%\int_0^T\int_{\Omega(t)}\dot \varphi(t)\eta(t) + \int_0^T\int_{\Omega(t)}\alpha(x,t)\sgrad\varphi(t) \sgrad \eta(t) + \int_0^T\int_{\Omega(t)}\sgrad \alpha(x,t)\sgrad \varphi(t) \eta(t)&= 0
%\end{align*}
%for all $\eta \in L^2_{H^1}$.
%\paragraph{Well-posedness}
Define the bilinear form $a(t;\varphi, \eta) = \int_{\Omega(t)}\alpha(x,t)\sgrad\varphi \sgrad \eta + \int_{\Omega(t)}\sgrad \alpha(x,t)\sgrad \varphi \eta$
which is clearly bounded and coercive on $H^1(\Omega(t))$.
%\begin{align*}
%a(t;\eta, \eta) %&= \int_{\Omega(t)}\alpha(x,t)|\sgrad \eta|^2 + \int_{\Omega(t)}\sgrad \alpha(x,t)\sgrad \eta \eta\\
%&\geq \epsilon\norm{\sgrad \eta}{L^2(\Omega(t))}^2 - \norm{\sgrad \alpha}{L^\infty(\Omega(t))}\int_{\Omega(t)}|\sgrad \eta| |\eta|\\
%&\geq \epsilon\norm{\sgrad \eta}{L^2(\Omega(t))}^2 - \delta \norm{\sgrad \alpha}{L^\infty(\Omega(t))}\norm{\sgrad \eta}{L^2(\Omega(t))}^2 - C_\delta\norm{\eta}{L^2(\Omega(t))}^2
%\end{align*}
%provided we pick $\delta$ from Young's inequality small enough.
 Split $a(t;\cdot,\cdot)$ into the forms $a_s(t;\varphi, \eta) := \int_{\Omega(t)}\alpha(x,t)\sgrad\varphi \sgrad \eta$ and $a_n(t;\varphi,\eta) :=  \int_{\Omega(t)}\sgrad \alpha(x,t)\sgrad \varphi \eta.$
One sees that $a_s(t;\eta,\eta) \geq 0$ and that both %$|a_n(t,\varphi, \eta)| \leq C\norm{\sgrad \varphi}{L^2(\Omega_0)}\norm{\eta}{L^2(\Omega_0)}$
$a_n(t;\cdot,\cdot)\colon H^1(\Omega(t)) \times L^2(\Omega(t)) \to \mathbb{R}$ and $a_s(t;\cdot,\cdot)\colon H^1(\Omega(t)) \times H^1(\Omega(t)) \to \mathbb{R}$ are bounded. Also, letting $\chi_j^t := \phi_t \chi_j^0$ where $\chi_j^0$ are the normalised eigenfunctions of $-\Delta_{\Omega_0}$, we have for $\eta \in \tilde C^1_{H^1} :=  \{u \mid u(t) = \sum_{j=1}^m \alpha_j(t)\chi_j^t, \text{ $m \in \mathbb{N}$, $\alpha_j \in AC([0,T])$ and $\alpha_j' \in L^2(0,T)$}\},$
\begin{align*}
\frac{d}{dt}a_s(t;\eta(t), \eta(t)) %&= \frac{d}{dt}\int_{\Omega(t)}\alpha(x,t)\sgrad\eta(t)\sgrad \eta(t)\\
%&= 2\int_{\Omega(t)}\alpha(t)\sgrad \dot \eta(t) \sgrad \eta(t) + \int_{\Omega(t)}\dot \alpha(t)|\sgrad \eta(t)|^2 + \alpha(t)\sgrad \cdot \mathbf{w} |\sgrad \eta(t)|^2 - 2D(\mathbf{w})\sgrad \eta(t) \sgrad \eta(t)\\
&= 2a_s(t;\dot \eta(t), \eta(t)) + r(t;\eta(t))
\end{align*}
where $r$ is such that $|r(t;\eta(t))| \leq C\norm{\eta(t)}{H^1(\Omega(t))}^2$ (see \cite[Lemma 2.1]{DziukElliottL2}, note that $\tilde \alpha \in C^1([0,T];C^1(\Omega_0))$ and thus $\alpha \in C^1_{H^1}$). Hence by \cite[Theorem 3.13]{AlpEllSti} we have the unique existence of $\varphi \in W(H^1, L^2)$. Rearranging the equation \eqref{eq:auxiliaryPDE} shows that $\alpha\slap \varphi \in L^2_{L^2}$. % Since then the equation holds in $L^2_{L^2}$, it holds pointwise a.e too.
 Since $\alpha$ is uniformly bounded by positive constants, it follows that $\slap \varphi \in L^2_{L^2}$.
\paragraph{The $L^\infty$ bound}Let $K := \norm{\xi}{L^\infty(\Omega_0)}$. Test the equation with $(\varphi-K)^+$:
%\begin{align*}
%0%&=\int_{\Omega(t)}\dot\varphi(t)(\varphi(t)-K)^+ - \int_{\Omega(t)}\alpha(t)\slap\varphi(t)(\varphi(t)-K)^+\\
%&=\frac{1}{2}\frac{d}{dt}\int_{\Omega(t)}(\varphi(t)-K)^+(\varphi(t)-K)^+ - \frac{1}{2}\int_{\Omega(t)}((\varphi(t)-K)^+)^2\sgrad \cdot \mathbf w + \int_{\Omega(t)}\sgrad(\alpha(t)(\varphi(t)-K)^+)\sgrad \varphi(t).
%&=\frac{1}{2}\frac{d}{dt}\norm{(\varphi(t)-K)^+}{L^2(\Omega(t))}^2 - \frac{1}{2}\int_{\Omega(t)}((\varphi(t)-K)^+)^2\sgrad \cdot \mathbf w + \int_{\Omega(t)}\sgrad\alpha(t)\sgrad \varphi(t)(\varphi(t)-K)^+\\
%&\quad + \int_{\Omega(t)}\alpha(t)\sgrad((\varphi(t)-K)^+)\sgrad \varphi(t).
%\end{align*}
%Rearranging gives
\begin{align*}
\frac{1}{2}\frac{d}{dt}\norm{(\varphi(t)-K)^+}{L^2(\Omega(t))}^2&  + \int_{\Omega(t)}\alpha(t)\sgrad((\varphi(t)-K)^+)\sgrad \varphi(t)\\
&\!\!\!\!= \frac{1}{2}\int_{\Omega(t)}((\varphi(t)-K)^+)^2\sgrad \cdot \mathbf w - \int_{\Omega(t)}\sgrad\alpha(t)\sgrad \varphi(t)(\varphi(t)-K)^+
\end{align*}
which becomes, through the use of Young's inequality with $\delta$,
%\begin{align*}
%\frac{1}{2}&\frac{d}{dt}\norm{(\varphi(t)-K)^+}{L^2(\Omega(t))}^2  + \epsilon\norm{\sgrad((\varphi(t)-K)^+)}{L^2(\Omega(t))}^2\\
%%&\leq \frac{\norm{\sgrad \cdot \mathbf w}{\infty}}{2}\norm{(\varphi(t)-K)^+}{L^2(\Omega(t)}^2 + \norm{\sgrad \alpha}{L^\infty}\int_{\Omega(t)}|\sgrad \varphi(t)(\varphi(t)-K)^+|\\
%&\leq \frac{\norm{\sgrad \cdot \mathbf w}{\infty}}{2}\norm{(\varphi(t)-K)^+}{L^2(\Omega(t)}^2 + \norm{\sgrad \alpha}{L^\infty}\int_{\Omega(t)}|\sgrad ((\varphi(t)-K)^+)(\varphi(t)-K)^+|
%%\tag{since $\sgrad \varphi((\varphi-K)^+) = \sgrad (\varphi-K)^+(\varphi-K)^+$}
%\\
%&\leq \frac{\norm{\sgrad \cdot \mathbf w}{\infty}}{2}\norm{(\varphi(t)-K)^+}{L^2(\Omega(t)}^2 + \norm{\sgrad \alpha}{L^\infty}(\delta\norm{\sgrad ((\varphi(t)-K)^+)}{L^2(\Omega(t))}^2 + C_\delta\norm{(\varphi(t)-K)^+}{L^2(\Omega(t))}^2)
%\end{align*}
%and with $\delta$ small enough we get
\begin{align*}
\frac{1}{2}\frac{d}{dt}\norm{(\varphi(t)-K)^+}{L^2(\Omega(t))}^2 %& + (\epsilon-\norm{\sgrad \alpha}{L^\infty}\delta)\norm{\sgrad((\varphi(t)-K)^+)}{L^2(\Omega(t))}^2\\
&\leq \left(\frac{C_{\mathbf w}}{2}+ \norm{\sgrad \alpha}{L^\infty}C_\delta\right)\norm{(\varphi(t)-K)^+}{L^2(\Omega(t))}^2.
\end{align*}
An application of Gronwall's inequality and noticing $(\varphi(0)-K)^+ = (\xi - \norm{\xi}{L^\infty})^+ = 0$ yields 
%\[\norm{(\varphi(t)-K)^+}{L^2(\Omega(t))}^2 \leq C(t)\norm{(\xi-K)^+}{L^2(\Omega_0)}^2.\]
%Since $K=\norm{\xi}{L^\infty}$, the right hand side is zero, hence 
$\varphi(t) \leq \norm{\xi}{L^\infty(\Omega_0)}$. Repeating this process with $(-\varphi(t)-K)^+$ allows us to conclude.% $\norm{\varphi}{L^\infty_{L^\infty}} \leq \norm{\xi}{L^\infty(\Omega_0)}$.
\paragraph{The inequality \eqref{eq:auxiliaryPDEInequality}}Multiplying the equation \eqref{eq:auxiliaryPDE} by $\slap \varphi$ and integrate: formally,
\begin{align}
\nonumber \int_0^t \int_{\Omega(\tau)}&\alpha|\slap \varphi|^2 %&= \int_0^t \int_{\Omega(\tau)}\dot \varphi \slap \varphi\\
= -\int_0^t \int_{\Omega(\tau)}\sgrad \dot \varphi \sgrad \varphi\\
\nonumber &=-\int_0^t \frac{1}{2}\frac{d}{d\tau}\int_{\Omega(\tau)}|\sgrad \varphi|^2 + \frac{1}{2}\int_0^t\int_{\Omega(\tau)}|\sgrad \varphi|^2\sgrad \cdot \mathbf w - \int_0^t \int_{\Omega(\tau)}D(\mathbf w)\sgrad \varphi \sgrad \varphi\\
&\leq \frac{1}{2}\int_{\Omega_0}|\sgrad \xi|^2 - \frac{1}{2}\int_{\Omega(t)}|\sgrad \varphi(t)|^2 + C_{\mathbf w}\int_0^t \int_{\Omega(\tau)}|\sgrad \varphi|^2.\label{eq:11}
\end{align}
See \cite[Lemma 2.1]{DziukElliottL2} or \cite{AlpEllSti2} for the definition of the matrix $D(\mathbf w)$. This calculation is merely formal because we have not shown that $\dot \varphi(t) \in H^1(\Omega(t))$; however the end result of the calculation is still valid by Lemma \ref{lem:density}. We also have by squaring \eqref{eq:auxiliaryPDE}, integrating and using \eqref{eq:11}:
\begin{align*}
\int_0^t \int_{\Omega(\tau)}(\dot \varphi(\tau))^2% &= \int_0^t \int_{\Omega(\tau)}\alpha^2(\slap \varphi)^2\\ 
&\leq \alpha_0\int_0^t \int_{\Omega(\tau)}\alpha(\slap \varphi)^2 \leq \frac{\alpha_0}{2}\int_{\Omega_0}|\sgrad \xi|^2 + \alpha_0C_{\mathbf w}\int_0^t \int_{\Omega(\tau)}|\sgrad \varphi|^2.
\end{align*}
Adding the last two inequalities then we obtain
\begin{align*}%\label{eq:aux1}
\int_0^t \int_{\Omega(\tau)}(\dot \varphi(\tau))^2+ \int_0^t \int_{\Omega(\tau)}\alpha|\slap \varphi|^2 + \frac{1}{2}\int_{\Omega(t)}|\sgrad \varphi(t)|^2&\leq \frac{1+\alpha_0}{2}\int_{\Omega_0}|\sgrad \xi|^2\\
&\quad  + C_{\mathbf w}(1+\alpha_0)\int_0^t \int_{\Omega(\tau)}|\sgrad \varphi|^2 .
\end{align*}
Gronwall's inequality can be used to deal with the last term on the right hand side.
\end{proof}
\begin{lem}\label{lem:density}
With $\varphi \in W(H^1,L^2)$ from the previous lemma, the following inequality holds:
\begin{align}
\int_0^t \int_{\Omega(\tau)}\alpha|\slap \varphi|^2 
&\leq \frac{1}{2}\int_{\Omega_0}|\sgrad \xi|^2 - \frac{1}{2}\int_{\Omega(t)}|\sgrad \varphi(t)|^2 + C_{\mathbf w}\int_0^t \int_{\Omega(\tau)}|\sgrad \varphi|^2.\label{eq:densityEquation}
\end{align}
\end{lem}
\begin{proof}
Let $C^\infty_{H^2} := \{\eta \mid \phi_{-(\cdot)}\eta(\cdot) \in C^\infty([0,T];H^2(\Omega_0))\}$. We start with a few preliminary results.
%\begin{enumerate}
%\item 
Let us show $C^\infty_{H^2} \subset W(H^2, H^1).$ Take $\eta \in C^\infty_{H^2}$ so that $\tilde \eta \in C^\infty([0,T];H^2(\Omega_0)) \subset \mathcal W(H^2, H^1)$. By smoothness of $\Phi^{(\cdot)}_0$, it follows that $\eta = \phi_{(\cdot)}\tilde \eta \in L^2_{H^2}$, and $\dot \eta = \md(\phi_{(\cdot)}\tilde \eta) = \phi_{(\cdot)}(\tilde \eta') \in L^2_{H^1}$ because $\tilde \eta' \in C^\infty([0,T];H^2(\Omega_0)) \subset L^2(0,T;H^1(\Omega_0))$. So $\eta \in W(H^2, H^1)$.

Let us also prove that 
%\begin{align*}
$C^\infty_{H^2} \subset W(H^2, L^2) \text{ is dense.}$
%\end{align*} 
Let $w \in W(H^2, L^2)$; then $\tilde w \in \mathcal{W}(H^2, L^2)$ since $\tilde w \in L^2(0,T;H^2(\Omega_0))$ by smoothness of $\Phi^{(\cdot)}_0$ and since $\tilde w' = \phi_{-(\cdot)}\dot w \in L^2(0,T;L^2(\Omega_0))$ (because $\dot w \in L^2_{L^2}$). By \cite[Lemma II.5.10]{Boyer} there exists $\tilde w_n \in C^\infty([0,T];H^2(\Omega_0))$ with $\tilde w_n \to \tilde w$ in $\mathcal{W}(H^2, L^2)$. Then, $w_n := \phi_{(\cdot)}\tilde w_n \in C^\infty_{H^2}$ (by definition) and
%\begin{align*}
\[\norm{w_n -  w}{W(H^2, L^2)}
%&= \norm{w_n -  w}{L^2_{H^2}} + \norm{\dot w_n -  \dot w}{L^2_{L^2}}\\
\leq C\left(\norm{\tilde w_n - \tilde w}{L^2(0,T;H^2(\Omega_0))}+ \norm{\tilde w_n'- \tilde w'}{L^2(0,T;L^2(\Omega_0))}\right) \to 0,\]
%\end{align*}
where we used the smoothness of $\Phi^{(\cdot)}_0$ and the reasoning behind Assumption 2.37 of \cite{AlpEllSti} (see also \cite[Theorem 2.33]{AlpEllSti}).% to get the penultimate line.
%\item Now let us prove that $W(H^2, L^2) \hookrightarrow C^0_{H^1}$. 
%\end{enumerate}

Given $\varphi \in W(H^2, L^2)$, by the density result, there exists $\varphi_n \in C^\infty_{H^2} \subset W(H^2, H^1)$ such that $\varphi_n \to \varphi$ in $W(H^2, L^2)$ with $\varphi_n$ satisfying \eqref{eq:densityEquation}:
\begin{align}
\int_0^t \int_{\Omega(\tau)}\alpha|\slap \varphi_n|^2 
&\leq \frac{1}{2}\int_{\Omega_0}|\sgrad \varphi_n(0)|^2 - \frac{1}{2}\int_{\Omega(t)}|\sgrad \varphi_n(t)|^2 + C_{\mathbf w}\int_0^t \int_{\Omega(\tau)}|\sgrad \varphi_n|^2.\label{eq:densityEquationN}
\end{align}
%that $\norm{\sgrad \varphi_n(t)}{L^2(\Omega(t))} \to \norm{\sgrad \varphi(t)}{L^2(\Omega(t))}$ for each $t$. 
We know that $\tilde \varphi_n \to \tilde \varphi$ in $\mathcal W(H^2, L^2)$ (this is just how we construct the sequence $\varphi_n$; see above), and $\mathcal W(H^2, L^2)  \hookrightarrow C^0([0,T];H^1(\Omega_0))$ \cite[Lemma II.5.14]{Boyer} implies $\varphi_n(t) \to \varphi(t)$ in $H^1(\Omega(t))$. %(because $\phi_t\colon H^1(\Omega_0) \to H^1(\Omega(t))$ is continuous). 
Now we can pass to the limit in every term in \eqref{eq:densityEquationN}.% which gives the desired result.
%Let us now prove that $W(H^1, L^2) \hookrightarrow C^0_{H^1}$. 
\end{proof}
\section{Well-posedness}
%Note that $\mathcal{E}$ is a maximal monotone operator on $\mathbb{R}$ (see \cite{chen} or Section 3 in \cite{RodriguesDarcy}). 
We can approximate $\mathcal{E}$ by $C^\infty$ bi-Lipschitz functions $\mathcal{E}_\epsilon$ such that (for example see \cite{RodriguesDarcy, RodriguesBook})
\begin{align}
\nonumber &\mathcal{E}_\epsilon \to \mathcal E \text{ uniformly in the compact subsets of $\mathbb{R} \backslash \{0\}$}\\
\nonumber &\mathcal{E}_\epsilon^{-1} \to \mathcal E^{-1} \text{ uniformly in the compact subsets of $\mathbb{R}$}\\%(see \cite[Proof of Theorem 4.1, p.~195]{RodriguesBook})
%\nonumber &\mathcal{E}_\epsilon \text{ is bi-Lipschitz}\\
\nonumber &\mathcal{E}_\epsilon(0) = 0\text{ and } \mathcal{E}_\epsilon = \mathcal{E} \text{ on }(-\infty, 0) \cup (\epsilon, \infty)\\
%\nonumber  &\mathcal{E}_\epsilon(r) = \begin{cases}
%r & \text{for $r \leq  0$}\\
%%\text{something} & \text{for $0 < r \leq \epsilon$}\\
%r+1 & \text{for $r > \epsilon$}
%\end{cases}\\
%\nonumber & 1\leq \mathcal{E}'_\epsilon(r) \leq 1+L_\epsilon\quad\text{for all $r \in \mathbb{R}$}\\
%\nonumber &\frac{1}{1+L_\epsilon} \leq (\mathcal{E}_\epsilon^{-1}(r))' \leq 1\quad\text{for all $r \in \mathbb{R}$} \\
\nonumber & 1\leq \mathcal{E}'_\epsilon(r) \leq 1+L_\epsilon\text{ and } ({1+L_\epsilon})^{-1} \leq (\mathcal{E}_\epsilon^{-1}(r))' \leq 1\quad\text{for all $r \in \mathbb{R}$} 
\end{align}
{(where $L_\epsilon=\mathcal{O}(1/\epsilon)$ is the Lipschitz constant of the approximation to the Heaviside function)}. %Note that the last property implies for all $r$, $s \in \mathbb R$,
%\begin{align}
%|\mathcal{E}_\epsilon^{-1}(r)-\mathcal{E}_\epsilon^{-1}(s)| &\leq |r-s|.%\label{eq:lipschitzPropertyOfEepsilonInverse}%\\
%%|\mathcal{E}_\epsilon^{-1}(r)| &\leq |r|\label{eq:boundednessOfEepsilonInverse}.
%\end{align}
We write $\mathcal{U} := \mathcal{E}^{-1}$ and $\mathcal{U}_\epsilon := \mathcal{E}^{-1}_\epsilon$. In order to prove Theorem \ref{thm:existenceLinfty}, that of the well-posedness of $L^\infty$ weak solutions given bounded data, we consider the following approximation of \eqref{eq:pde}.% using $\mathcal{E}_\epsilon$ and $\mathcal{U}_\epsilon$.%: find a $u_\epsilon$ satisfying
%\begin{equation}\label{eq:1}
%\begin{aligned}
%\md (\mathcal{E}_\epsilon(u_\epsilon(t))) - \slap u_\epsilon(t) + \mathcal{E}_\epsilon(u_\epsilon(t)) \sgrad \cdot \mathbf{w}&= f\\
%u_\epsilon(0) &= u_{0}.
%\end{aligned}
%\end{equation}
%Using $\mathcal{U}_\epsilon$, we can formulate the following.
\begin{defn}\label{defn:epsilonProblem}
%With $v_0 \in \mathcal{E}(u_0)$ and a sequence $v_{0\epsilon} \to v_0$ as $\epsilon \to 0$, 
Find for each $\epsilon > 0$ a function $e_\epsilon \in W(H^1, H^{-1})$ such that
\begin{equation}\label{eq:2}
\begin{aligned}
\md e_\epsilon- \slap (\mathcal{U}_\epsilon e_\epsilon) + e_\epsilon \sgrad \cdot \mathbf{w}&= f\quad\text{in $L^2_{H^{-1}}$}\\
e_\epsilon(0) &= e_{0}.
\end{aligned}\tag{$\mathbf{P}_\epsilon$}
\end{equation}
\end{defn}
\begin{theorem}Given $f \in L^2_{H^{-1}}$ and $e_0 \in L^2(\Omega_0)$, the problem \eqref{eq:2} has a weak solution $e_\epsilon \in W(H^1, H^{-1})$.
\end{theorem}
\begin{proof}
Using the chain rule on the nonlinear term leads us to consider for fixed $w \in W(H^1, H^{-1})$
\begin{equation}\label{eq:fpEquation}
\begin{aligned}
\langle \md(Sw), \eta \rangle_{L^2_{H^{-1}}, L^2_{H^1}} +(\mathcal{U}_\epsilon'(w)\sgrad (Sw),\sgrad \eta)_{L^2_{L^2}} + (Sw, \eta \sgrad \cdot \mathbf w)_{L^2_{L^2}} &= \langle f,\eta\rangle_{L^2_{H^{-1}}, L^2_{H^1}}\\
%\int_0^T \langle \md Sw(t), \eta(t) \rangle + \int_{\Omega(t)}\mathcal{U}_\epsilon'(w(t))\sgrad Sw(t)\cdot \sgrad \eta(t) + \int_{\Omega(t)} Sw(t) \eta(t)\sgrad \cdot \mathbf{w}  &= \int_0^T\int_{\Omega(t)}f(t)\eta(t)\\
Sw(0) &= e_{0}.
\end{aligned}\tag{$\mathbf{P}(w)$}
\end{equation}
If $S$ denotes the solution map of \eqref{eq:fpEquation} that takes $w \mapsto Sw$, then we seek a fixed point of $S$. %We want to use the Schauder--Tikhonov fixed point theorem (Theorem \ref{thm:Schauder}). 
First, note that since the bilinear form involving the surface gradients
%\[a(t;u,v) :=\int_{\Omega(t)}\mathcal{U}_\epsilon'(w(t))\sgrad u(t)\cdot \sgrad v(t) \]
is bounded and coercive, the solution $Sw \in W(H^1, H^{-1})$ of \eqref{eq:fpEquation} does indeed exist by \cite[Theorem 3.6]{AlpEllSti}, and moreover, it satisfies the estimate
\begin{equation}\label{eq:estimateFP1}
\norm{Sw}{W(H^1, H^{-1})} \leq C\left(\norm{f}{L^2_{H^{-1}}} + \norm{u_0}{L^2(\Omega_0)}\right) =: C_*
\end{equation}
where the constant $C$ does not depend on $w$ because $\mathcal{U}_\epsilon'(w(t))$ is uniformly bounded from below (in $w$). Then the set $E := \{ w \in W(H^1, H^{-1}) \mid w(0) = e_{0}, \norm{w}{W(H^1, H^{-1})} \leq C_*\},$ which is a closed, convex, and bounded subset of $X:=W(H^1, H^{-1})$, is such that $S(E) \subset E$ by \eqref{eq:estimateFP1}. %Also, because of the aforementioned properties of $E$ and because $X$ is reflexive, $E$ is weakly compact (see II.C.1 in \url{http://books.google.co.uk/books?id=grQ9GujyukMC&pg=PA49&lpg=PA49&dq=weakly+compact+set&source=bl&ots=qJIGnWdRWx&sig=9OtcOxkBFz5OjZ75oCAwyt6vPlw&hl=en&sa=X&ei=XoEfU-GaNcLwhQfgmYD4DQ&ved=0CC0Q6AEwADgK#v=onepage&q=weakly%20compact%20set&f=false}). 
%It seems to suffice (see \url{http://books.google.co.uk/books?id=peZHAAAAQBAJ&pg=PA65&lpg=PA65&dq=tikhonov+fixed+point+theorem+%22weakly+continuous%22&source=bl&ots=9kqwXB1B-E&sig=w6Nw0_lJxUH4E8Z1z7fTD3z7Kfk&hl=en&sa=X&ei=sGAgU_SIPM7A7Aa294CYCw&ved=0CEEQ6AEwBA#v=onepage&q=tikhonov%20fixed%20point%20theorem%20%22weakly%20continuous%22&f=false}) to show that $S \colon E \to E$ is weakly continuous. So let $w_n \weaklyto w$ in $E$. We have to show that $Sw_n \weaklyto Sw$.  
We now show that $S$ is weakly continuous. Let $w_n \weaklyto w$ in $W(H^1, H^{-1})$ with $w_n \in E$. 
From the estimate \eqref{eq:estimateFP1}, we know that $Sw_n$ is bounded in $W(H^1, H^{-1})$, so for a subsequence
\begin{align*}
Sw_{n_j} &\weaklyto \chi \quad \text{in $W(H^1, H^{-1})$}\\
Sw_{n_j} &\to \chi \quad \text{in $L^2_{L^2}$}%\\
%Sw_{n_{j_k}} &\to \chi \quad \text{a.e.}
\end{align*}
by the compact embedding of Lemma \ref{lem:compactEmbedding}. Now we show that $\chi = Sw$. Due to $W(H^1, H^{-1}) \cts C^0_{L^2}$, $Sw_{n_j} \weaklyto \chi$ in $C^0_{L^2}$. This implies $Sw_{n_j}(0) \weaklyto \chi(0)$ in $L^2(\Omega_0)$ (to see this consider for arbitrary $f \in L^2(\Omega_0)$ the functional $G \in (C^0_{L^2})^*$ defined $G(u_n) = \int_{\Omega_0}fu_n(0)$). Since $Sw_{n_j}(0) = e_{0}$, it follows that 
\begin{equation}\label{eq:6}
\chi(0) = e_{0}.
\end{equation}
On the other hand, since $w_n$ are weakly convergent in $W^1(H^1, H^{-1})$, they are bounded in the same space. Now, $W(H^1, H^{-1}) \compact L^2_{L^2}$, hence $w_{n} \to w$ in $L^2_{L^2}$. It follows that the subsequence $w_{n_j} \to w$ in $L^2_{L^2}$ too, and so there is a subsequence such that for almost every $t \in [0,T]$, $w_{n_{j_k}}(t) \to w(t)$ a.e. in $\Omega(t)$. By continuity, for a.a. $t$, $\mathcal{U}_\epsilon'(w_{n_{j_k}}(t))\sgrad \eta(t) \to \mathcal{U}_\epsilon'(w(t))\sgrad \eta(t)$ a.e., and also we have $|\mathcal{U}_\epsilon'(w_{n_{j_k}})\sgrad \eta| \leq |\sgrad \eta|$ with the right hand side in $L^2_{L^2}$. Thus we can use the dominated convergence theorem (Theorem \ref{thm:DCT}) which tells us that $\mathcal{U}_\epsilon'(w_{n_{j_k}})\sgrad \eta \to \mathcal{U}_\epsilon'(w)\sgrad \eta$ in $L^2_{L^2}$. Now we pass to the limit in the equation \eqref{eq:fpEquation} with $w$ replaced by $w_{n_{j_k}}$ to get
%\[\int_0^T \langle \md Sw_{n_{j_k}}(t), \eta(t) \rangle +  \int_{\Omega(t)}\mathcal{U}_\epsilon'(w_{n_{j_k}}(t))\sgrad Sw_{n_{j_k}}(t) \sgrad \eta(t) +  \int_{\Omega(t)}Sw_{n_{j_k}}(t) \eta(t)\sgrad \cdot \mathbf{w}  = \int_0^T\int_{\Omega(t)}f(t)\eta(t),
%\]
%for a subsequence $n_{j_k}$ in all terms to receive
\[\int_0^T \langle \md \chi(t), \eta(t) \rangle +  \int_{\Omega(t)}\mathcal{U}_\epsilon'(w(t))\sgrad \chi(t) \sgrad \eta(t) +  \int_{\Omega(t)}\chi(t) \eta(t)\sgrad \cdot \mathbf{w}  = \int_0^T\langle f(t), \eta(t)\rangle
\]
which, along with \eqref{eq:6}, shows that $\chi = Sw$, so $Sw_{n_j} \rightharpoonup Sw$.  However, we have to show that the whole sequence converges, not just a subsequence. Let $x_n = Sw_n$ and equip the space $X=W(H^1, H^{-1})$ with the weak topology. Let $x_{n_m}=Sw_{n_m}$ be a subsequence. By the bound of $S$, it follows that $x_{n_m}$ is bounded, hence it has a subsequence such that
\begin{align*}
x_{n_{m_l}} \weaklyto x^* \text{ in $X$} \qquad \text{and} \qquad x_{n_{m_l}} \to x^* \text{ in $L^2_{L^2}$}.%\\
\end{align*}
By similar reasoning as before, we identify $x^* = Sw$, and Theorem \ref{thm:topologicalSubsequencesResult} below tells us that indeed $x_n = Sw_n \weaklyto Sw$. Then by the Schauder--Tikhonov fixed point theorem \cite[Theorem 1.4, p.~118]{Etienne}, $S$ has a fixed point.
\end{proof}
\begin{theorem}\label{thm:topologicalSubsequencesResult}
Let $x_n$ be a sequence in a topological space $X$ such that every subsequence $x_{n_j}$ has a subsequence $x_{n_{j_k}}$ converging to $x\in X$. Then the full sequence $x_n$ converges to $x$.
\end{theorem}
\subsection{Uniform estimates}
We set $u_\epsilon = \mathcal{U}_\epsilon(e_\epsilon)$. Below we denote by $M$ a constant such that $\norm{u_0}{L^\infty(\Omega_0)} \leq M$.
% and $M := \norm{u_0}{L^\infty(\Omega_0)}$.
%We rewrite the above equation as
%\begin{equation}\label{eq:3}
%\begin{aligned}
%\md (\mathcal{E}_\epsilon(u_\epsilon(t)))- \slap u_\epsilon(t) + \mathcal{E}_\epsilon(u_\epsilon(t)) \sgrad \cdot \mathbf{w} &= f \qquad \text{in $L^2_{H^{-1}}$}\\
%u_\epsilon(0) &= u_{0}
%\end{aligned}
%\end{equation}
%\subsubsection{Estimates when $f \in L^\infty$}
\begin{lem}\label{lem:energyEstimateLinfty}
The following bound holds independent of $\epsilon$:
%\[\norm{u_\epsilon}{L^\infty_{L^\infty}} \leq C(T,M,\mathbf w, f).\]
%where \[C(T,M,\mathbf w, f) = e^{\norm{\sgrad \cdot \mathbf w}{\infty} T}\left(T\norm{f}{L^\infty_{L^\infty}} + (1+M)\right).\]
\begin{align*}
%\norm{u_\epsilon}{L^\infty_{L^\infty}} &\leq e^{\norm{\sgrad \cdot \mathbf w}{\infty} T}\left(T\norm{f}{L^\infty_{L^\infty}} +\norm{u_0}{L^\infty(\Omega_0)} + 1\right)\\
\norm{u_\epsilon}{L^\infty_{L^\infty}} + \norm{\mathcal{E}_\epsilon(u_\epsilon)}{L^\infty_{L^\infty}} &\leq 2e^{\norm{\sgrad \cdot \mathbf w}{\infty} T}\left(T\norm{f}{L^\infty_{L^\infty}} +\norm{u_0}{L^\infty(\Omega_0)} + 1\right) + 1.
\end{align*}
\end{lem}
\begin{proof}
We substitute $w(t) = e^{-\lambda t}e_\epsilon(t)$ in \eqref{eq:2} %to get
%\begin{equation}\label{eq:4a}
%\md (e^{\lambda t}w(t)) - \slap (\mathcal{U}^{\epsilon}(e^{\lambda t}w(t)))+ e^{\lambda t}w(t)\sgrad \cdot \mathbf{w}= f(t)
%\end{equation}
and use $\md (e^{\lambda t}w(t)) = \lambda e^{\lambda t}w(t) + e^{\lambda t}\dot w(t)$ to get
\[\dot w(t) - e^{-\lambda t}\slap (\mathcal{U}^{\epsilon}(e^{\lambda t}w(t)))+\lambda w(t)+ w(t)\sgrad \cdot \mathbf{w}= e^{-\lambda t}f(t).\]
%(To see why we can divide through by $e^{\lambda t}$, we write the equality \eqref{eq:4a} in terms of the duality pairing with a test function $\eta$, then transform $\eta(t) \to e^{-\lambda t}\eta(t)$, which by smoothness of $e^{-\lambda t}$, lies in $L^2_{H^1}$.)
%Testing with $\eta$:% ({this is OK because the equality in $L^2_{H^{-1}}$ can be written as an equality for a.e. $t$}):
%\begin{align}
%\langle \dot w(t),\eta(t)\rangle_{\Vmt,\Vt} &+ \int_{\Omega(t)}e^{-\lambda t}\sgrad (\mathcal{U}^{\epsilon}(e^{\lambda t}w(t)))\sgrad \eta(t) +\int_{\Omega(t)}(\lambda+\sgrad \cdot \mathbf w) w(t)\eta(t)= e^{-\lambda t}\int_{\Omega(t)}f(t)\eta(t)\label{eq:5a}.
%\end{align}
Let $\alpha = \norm{f}{L^\infty_{L^\infty}}$ and $\beta = \norm{e_{0}}{L^\infty(\Omega_0)}$ and define $v(t) = \alpha t + \beta$. Note that $\dot v(t) = \alpha$ and $v(0) = \beta.$ %and $v$ satisfies $\langle \dot v(t), \eta(t) \rangle  = \int_{\Omega(t)}\alpha \eta(t).$ 
Subtracting $\dot v(t)$ from the above and testing with
%\begin{align}
%\nonumber \langle \dot w(t)-\dot v(t),\eta(t)\rangle_{\Vmt,\Vt} 
%+ \int_{\Omega(t)}e^{-\lambda t}\sgrad ((\mathcal{E}^{\epsilon})^{-1}(e^{\lambda t}w(t)))\sgrad \eta(t) &+\int_{\Omega(t)}(\lambda+\sgrad \cdot \mathbf w) w(t)\eta(t)\\&= \int_{\Omega(t)}(e^{-\lambda t}f(t)-\alpha) \eta(t).
%\end{align}
$(w(t)-v(t))^+$, we get
\begin{align}
\nonumber &\langle \dot w(t)-\dot v(t),(w(t)-v(t))^+\rangle_{H^{-1}(\Omega(t)),H^1(\Omega(t))} + \int_{\Omega(t)}e^{-\lambda t}\sgrad (\mathcal{U}^{\epsilon}(e^{\lambda t}w(t)))\sgrad (w(t)-v(t))^+ \\
&+\int_{\Omega(t)}(\lambda+\sgrad \cdot \mathbf w) w(t)(w(t)-v(t))^+= \int_{\Omega(t)}(e^{-\lambda t}f(t)-\alpha)(w(t)-v(t))^+\label{eq:5ab}.
\end{align}
Note that
%$\sgrad (\mathcal{U}_\epsilon(e^{\lambda t}w(t))) = \mathcal{U}_\epsilon'(e^{\lambda t}w(t))e^{\lambda t}\sgrad w(t)$, which implies that
$e^{-\lambda t}\sgrad (\mathcal{U}_\epsilon(e^{\lambda t}w(t)))\sgrad (w(t)-v(t))^+ %&=\int_{\Omega(t)}\mathcal{U}_\epsilon'(e^{\lambda t}w(t))\sgrad w(t)\sgrad (w(t)-v(t))^+\\
=\mathcal{U}_\epsilon'(e^{\lambda t}w(t))|\sgrad (w(t)-v(t))^+|^2
$ because $\sgrad v(t) = 0$.
Set $\lambda := \norm{\sgrad \cdot \mathbf w}{L^\infty}$, then the last term on the LHS of \eqref{eq:5ab} is non-negative because 
%\begin{align*}
%w(t)(w(t)-v(t))^+ = \begin{cases}
%w(t)(w(t)-v(t)) \geq 0 &\text{if $w(t) - v(t) > 0$}\\
%0&\text{if $w(t) - v(t) \leq 0$}
%\end{cases}
%\end{align*}
%because 
if $w > v$, $w > 0$ since $v \geq 0$. So we can throw away that and the gradient term to find
\begin{equation*}%\label{eq:mp1a}
\langle \dot w(t)- \dot v(t),(w(t)-v(t))^+\rangle_{H^{-1}(\Omega(t)),H^1(\Omega(t))} %+ \frac{1}{1+L_\epsilon}\int_{\Omega(t)}|\sgrad (w(t)-v(t))^+|^2  
\leq  \int_{\Omega(t)}(e^{-\lambda t}f(t)-\alpha)(w(t)-v(t))^+.
\end{equation*}
Integrating this and using Lemma \ref{lem:ibpWithMaxFunction}, we find
\begin{align*}
\frac{1}{2}\int_{\Omega(T)}((w(t)-v(t))^{+})^2   %+ \frac{1}{1+L_\epsilon}\int_0^T\int_{\Omega(t)}|\sgrad (w(t)-v(t))^+|^2\\
% &\leq  \int_0^T\int_{\Omega(t)}(e^{-\lambda t}f(t)-\alpha)(w(t)-v(t))^+ \\
% &\quad+\frac{1}{2}\int_0^T \int_{\Omega(t)}((w(t)-v(t))^+)^2\sgrad \cdot \mathbf w + \frac{1}{2}\int_{\Omega_0}((w(0)-v(0))^+)^2\\
&\leq  \frac{1}{2}\norm{\sgrad \cdot \mathbf w}{}\int_0^T \int_{\Omega(t)}((w(t)-v(t))^+)^2
\end{align*}
since $e^{-\lambda t}f(t)-\alpha = e^{-\lambda t}f(t) - \norm{f(t)}{L^\infty(\Omega(t))}\leq 0$ and $w(0)-v(0)=e_{0}-\norm{e_0}{L^\infty(\Omega_0)} \leq 0$. The use of Gronwall's inequality gives %us $\norm{(w(t)-v(t))^+}{L^2(\Omega(t))} = 0$ and
% old
%Note that
%\begin{align}
%\nonumber \frac{d}{dt}\int_{\Omega(t)}(w-v)^+(w-v)^+ &= 2\langle \md (w-v)^+, (w-v)^+ \rangle_{\Vmt, \Vt} + \int_{\Omega(t)}((w-v)^+)^2\sgrad \cdot \mathbf w\\
%&= 2\langle \md (w-v), (w-v)^+ \rangle_{\Vmt, \Vt} + \int_{\Omega(t)}((w-v)^+)^2\sgrad \cdot \mathbf w\label{eq:mp2}
%\end{align}
%because
%\begin{align*}
%\md (w-v)^+ = \begin{cases}
%\md (w-v) &: w-v > 0\\
%0&:w-v\leq 0
%\end{cases}
%\end{align*}
%which means that 
%\begin{align*}
%\langle \md (w-v)^+, (w-v)^+ \rangle_{\Vmt, \Vt} =
%\begin{cases}
%\langle \md(w-v), (w-v)^+ \rangle_{\Vmt, \Vt} &: w-v > 0\\
%0 = \langle \md (w-v), (w-v)^+ \rangle_{\Vmt, \Vt} &: w-v \leq 0\\
% \end{cases}
%\end{align*}
%with the last equality holding because $\langle \dot w, 0\rangle = 0$.
%So, using \eqref{eq:mp2} in \eqref{eq:mp1a} gives
%\begin{align*}
%&\frac{1}{2}\frac{d}{dt}\int_{\Omega(t)}((w(t)-v(t))^+)^2+ \frac{1}{1+L_\epsilon}%\int_{\Omega(t)}|\sgrad (w(t)-v(t))^+|^2\\
%&\leq  \frac{1}{2}\int_{\Omega(t)}((w(t)-v(t))^+)^2\sgrad \cdot \mathbf w+\int_{\Omega(t)}%(e^{-\lambda t}f(t)-\alpha)(w(t)-v(t))^+\\
%&\leq  \left(\norm{\sgrad \cdot \mathbf w}{L^\infty}\frac{1}{2}\right)\int_{\Omega(t)}((w(t)-%v(t))^+)^2 
%\end{align*}
%since $e^{-\lambda t}f(t)-\alpha = e^{-\lambda t}f(t) - \norm{f(t)}{L^\infty(\Omega(t))}\leq 0$. We can discount the second term on the left hand side and the use of Gronwall's inequality gives us
%\begin{align*}
%\norm{(w(t)-v(t))^+}{L^2(\Omega(t))}^2 &\leq C_{\mathbf w}\norm{(w(0)-v(0))^+}{L^2(\Omega_0)}%^2\\
%&= C_{\mathbf w}\norm{(e_{0}-\beta)^+}{L^2(\Omega_0)}^2\\
%&= 0
%\end{align*}
%this shows that 
$w(t) \leq T\norm{f}{L^\infty_{L^\infty}} + (1+M)$
almost everywhere on $\Omega(t)$. 
So we have shown that
%\begin{equation}\label{eq:null1}
{for all $t \in [0,T] \backslash N_1$, $w(t,x) \leq C$ for all $x \in \Omega(t)\backslash M^t_1$, where $\mu(N_1)=\mu(M^t_1) = 0.$}
%\end{equation}
A similar argument yields
%\begin{equation}\label{eq:null2}
for all $t \in [0,T] \backslash N_2$, $w(t,x) \geq -C$ for all $x \in \Omega(t)\backslash M^t_2$, where $\mu(N_2)=\mu(M^t_2) = 0.$
%\end{equation}
%Taking \eqref{eq:null1} and \eqref{eq:null2} 
Taking these statements together tells us that for all $t \in [0,T]\backslash N$, $|w(t,x)| \leq C$ on $\Omega(t)\backslash M^t$ where $N=N_1 \cup N_2$ and $M^t = M^t_1 \cup M^t_2$ have measure zero. This gives $\norm{w}{L^\infty_{L^\infty}} \leq T\norm{f}{L^\infty_{L^\infty}} + (1+M)$. 
%This gives $\norm{w(t)}{L^\infty(\Omega(t))} \leq T\alpha + \beta$ for every $t \in [0,T]$. 
From this and $u_\epsilon= \mathcal{U}_\epsilon(e^{\lambda (\cdot)}w(\cdot)) \leq e^{\lambda T}|w|$, we obtain the bound on $u_\epsilon$.
%\[\norm{u_\epsilon}{L^\infty_{L^\infty}} = \norm{\mathcal{U}_\epsilon(e^{\lambda (\cdot)}w(\cdot))}{L^\infty_{L^\infty}} %\leq \norm{e^{\lambda (\cdot)}w(\cdot)}{L^\infty_{L^\infty}} 
%\leq e^{\lambda T}\norm{w}{L^\infty_{L^\infty}} \leq e^{\norm{\sgrad \cdot \mathbf w}{\infty} T}\left(T\norm{f}{L^\infty_{L^\infty}} + (1+M)\right).\]
%by \eqref{eq:lipschitzPropertyOfEepsilonInverse}. 
The bound on $\mathcal{E}_\epsilon(u_\epsilon)$ follows from $\mathcal{E}_\epsilon(u_\epsilon) \leq 1 + |u_\epsilon|$.
\end{proof}
\begin{lem}
The following bound holds independent of $\epsilon$:
\begin{equation}\label{eq:boundOnGradientAndTimeDer}
\norm{\sgrad u_\epsilon}{L^2_{L^2}} + \norm{\md (\mathcal{E}_\epsilon u_\epsilon)}{L^2_{H^{-1}}}\leq C(T,\Omega,M,\mathbf w, f).
\end{equation}
%\[\norm{\md (\mathcal{E}_\epsilon u_\epsilon)}{L^2_{H^{-1}}} \leq C(T,\Omega,M,\mathbf w, f).\]
\end{lem}
\begin{proof}
Testing with $\mathcal{E}_\epsilon(u_\epsilon)$ in \eqref{eq:2}, %gives
%\[\langle \md (\mathcal{E}_\epsilon(u_\epsilon(t))), \mathcal{E}_\epsilon(u_\epsilon(t)) \rangle_{\Vmt, \Vt} + \int_{\Omega(t)} \sgrad u_\epsilon(t)\sgrad(\mathcal{E}_\epsilon(u_\epsilon(t))) + \int_{\Omega(t)}(\mathcal{E}_\epsilon(u_\epsilon(t)))^2 \sgrad \cdot \mathbf{w}= \int_{\Omega(t)}f(t)\mathcal{E}_\epsilon(u_\epsilon(t))\]
using $\sgrad u_\epsilon \sgrad(\mathcal{E}_\epsilon(u_\epsilon)) = (\mathcal{E}_\epsilon)'(u_\epsilon)|\sgrad u_\epsilon|^2 \geq |\sgrad u_\epsilon|^2$, %gives us
%\begin{equation*}
%\frac{1}{2}\frac{d}{dt}\norm{\mathcal{E}_\epsilon(u_\epsilon(t))}{L^2(\Omega(t))}^2+ \int_{\Omega(t)} |\sgrad u_\epsilon(t)|^2 + \frac{1}{2}\int_{\Omega(t)}(\mathcal{E}_\epsilon(u_\epsilon(t)))^2 \sgrad \cdot \mathbf{w}\leq \int_{\Omega(t)}f(t)\mathcal{E}_\epsilon(u_\epsilon(t))
%\end{equation*}
%Noting that, by the previous estimate,
%\begin{align*}
%\int_0^T\int_{\Omega(t)}(\mathcal{E}_\epsilon(u_\epsilon(t)))^2 \sgrad \cdot \mathbf{w} &\leq %C_1(T,M,\mathbf w, f)
%\end{align*}
%and
%\begin{align*}
%\int_0^T\int_{\Omega(t)}f(t)\mathcal{E}_\epsilon(u_\epsilon(t)) \leq \norm{f}{L^2_H}\norm{\mathcal{E}_\epsilon(u_\epsilon)}{L^2_H} \leq \norm{f}{L^2_H}C_2(T,M,\mathbf w, f),
%\end{align*}
integrating over time and using the previous estimate, we find
\begin{align*}
\frac{1}{2}\norm{\mathcal{E}_\epsilon(u_\epsilon(T))}{L^2(\Omega(T))}^2+ \int_0^T\int_{\Omega(t)} |\sgrad u_\epsilon(t)|^2 %&\leq \frac{1}{2}\norm{\mathcal{E}_\epsilon(u_{0})}{L^2(\Omega_0)}^2 + C_1(T,M,\mathbf w, f)\\
&\leq \frac{1}{2}(1+M)^2|\Omega_0| + C_1(T,M,\mathbf w, f).
\end{align*}
%\end{proof}
%We can also obtain a bound on the derivative. %$\norm{\md (\mathcal{E}_\epsilon u_\epsilon)}{L^2_{H^{-1}}}$; see Lemma \ref{lem:estimateOnDerivative} in the Appendix. 
%\begin{lem}\label{lem:estimateOnDerivative}The following bound holds independent of $\epsilon$:
%\end{lem}
%\begin{proof}
The bound on the time derivative follows by taking supremums.
\end{proof}
%The bounds imply for a subsequence (which we have not relabelled) the convergences
%\begin{equation*}
%\begin{aligned}
%u_\epsilon &\weaklyto u &&\text{in $L^2_{H^1}$}\\
%\mathcal{E}_\epsilon(u_\epsilon) &\weaklyto \chi &&\text{in %$L^2_{L^2}$}\\
%\md (\mathcal{E}_\epsilon(u_\epsilon)) &\weaklyto \dot \chi &&\text{in $L^2_{H^{-1}}$}
%\end{aligned}
%\end{equation*}
%This is not sufficient for concluding that $\chi = \mathcal{E}(u)$. Also, we remark that we cannot use Lions--Aubin here; the first convergence does not give us weak convergence of $\mathcal{E}_\epsilon(u_\epsilon)$ in $L^2_{H^1}$, and the last convergence does not give us weak convergence of $\dot u_\epsilon$ in $L^2_{H^{-1}}$; this is because we do not get bounds uniform in $\epsilon.$ So we need another \emph{a priori} estimate. 
\begin{lem}Define $\tilde u_\epsilon = \phi_{-(\cdot)}u_\epsilon$. The following limit holds uniformly in $\epsilon$:
\begin{align*}
\lim_{h \to 0}\int_0^{T-h}\int_{\Omega_0}|\tilde{u}_\epsilon(t+h) - \tilde{u}_\epsilon(t)| &= 0.
\end{align*}
\end{lem}
\begin{proof}
We follow the proof of Theorem A.1 in \cite{Blanchard} here. Fix $h \in (0,T)$ and %let $t \in (0, T-h)$
consider
\begin{align}
\nonumber \int_0^{T-h}(&\mathcal{E}_\epsilon(\tilde u_\epsilon(t+h))-\mathcal{E}_\epsilon(\tilde{u}_\epsilon(t)), \tilde{u}_\epsilon(t+h) - \tilde{u}_\epsilon(t))_{L^2(\Omega_0)}\;\mathrm{d}t\\
\nonumber &= \int_0^{T-h}\int_t^{t+h}\frac{d}{d\tau}(\mathcal{E}_\epsilon(\tilde{u}_\epsilon(\tau)),\tilde{u}_\epsilon(t+h) - \tilde{u}_\epsilon(t))_{L^2(\Omega_0)}\;\mathrm{d}\tau\;\mathrm{d}t\\
%\nonumber &= \int_0^{T-h}\int_t^{t+h}\langle (\mathcal{E}_\epsilon(\tilde{u}_\epsilon(\tau)))',\tilde{u}_\epsilon(t+h) - \tilde{u}_\epsilon(t)\rangle_{H^{-1}(\Omega_0),H^1(\Omega_0)}\;\mathrm{d}\tau\;\mathrm{d}t\\
%\nonumber &\leq \int_0^{T-h}\left(\int_t^{t+h}\norm{(\mathcal{E}_\epsilon(\tilde{u}_\epsilon(\tau)))'}{H^{-1}(\Omega_0)}\;\mathrm{d}\tau\right)\norm{\tilde{u}_\epsilon(t+h) - \tilde{u}_\epsilon(t)}{H^{1}(\Omega_0)}\;\mathrm{d}t\\
%Now we use Cauchy-Schwarz
%\[\int_t^{t+h}\norm{(\mathcal{E}_\epsilon(\tilde{u}_\epsilon(\tau)))'}{\Vms}\;\mathrm{d}\tau \leq \sqrt{h}\norm{(\mathcal{E}_\epsilon(\tilde{u}_\epsilon))'}{L^2(t,t+h;\Vms)}\leq \sqrt{h}\norm{(\mathcal{E}_\epsilon(\tilde{u}_\epsilon))'}{L^2(0,T;\Vms)}\]
%and proceed:
%\begin{align}
%\nonumber \nonumber \int_0^{T-h}(&\mathcal{E}_\epsilon(\tilde u_\epsilon(t+h))-\mathcal{E}_\epsilon(\tilde{u}_\epsilon(t)),\tilde{u}_\epsilon(t+h) - \tilde{u}_\epsilon(t))_{L^2(\Omega_0)}\\
\nonumber &\leq \sqrt{h}\norm{(\mathcal{E}_\epsilon(\tilde{u}_\epsilon))'}{L^2(0,T;H^{-1}(\Omega_0))}\int_0^{T-h}(\norm{\tilde{u}_\epsilon(t+h)}{H^{1}(\Omega_0)} + \norm{\tilde{u}_\epsilon(t)}{H^{1}(\Omega_0)})\;\mathrm{d}t\\
\nonumber &\leq C_1(T,\Omega, M, \mathbf w, f)\sqrt{h}\norm{(\mathcal{E}_\epsilon(\tilde{u}_\epsilon))'}{L^2(0,T;H^{-1}(\Omega_0))}\tag{by the uniform estimates}\\
%\nonumber &\leq C_2(T,\Omega, M, \mathbf w, f)\sqrt{h}\norm{\md(\phi_t(\mathcal{E}_\epsilon(\tilde{u}_\epsilon)))}{L^2_{H^{-1}}}\\
\nonumber &\leq C_2(T,\Omega, M, \mathbf w, f)\sqrt{h}\norm{\md(\mathcal{E}_\epsilon({u}_\epsilon))}{L^2_{H^{-1}}}\tag{see the proof of Theorem 2.33 in \cite{AlpEllSti}}\\
&\leq C_3(T,\Omega, M, \mathbf w, f)\sqrt{h},\label{eq:5d}
\end{align}
with the last inequality by \eqref{eq:boundOnGradientAndTimeDer}. Now, since the $\mathcal{U}_\epsilon'$ are uniformly bounded above, they are uniformly equicontinuous.
%Now, since the $\mathcal{U}_\epsilon$ are continuous, they are uniformly continuous on the compact interval $[-Me^{\lambda T},Me^{\lambda T}]$ (the interval in which the $u_\epsilon$ lie). Also, the Lipschitz constant of $\mathcal{U}_\epsilon$ is 1, which does not depend on $\epsilon$, so the $\mathcal{U}_\epsilon$ are equicontinuous. 
Therefore, for fixed $\delta$, there is a $\sigma_\delta$ (depending solely on $\delta$) such that 
\begin{equation}\label{eq:equicts}
\text{if $|y-z| < \sigma_\delta$, then $|\mathcal{U}_\epsilon(y) - \mathcal{U}_\epsilon(z)|  < \delta$}\qquad\text{ for \emph{any} $\epsilon$}.
\end{equation}
So in the set $\{|\tilde{u}_\epsilon(t+h) - \tilde u_\epsilon(t)| > \delta\} = \{|\mathcal{U}_\epsilon(\mathcal{E}_\epsilon(\tilde{u}_\epsilon(t+h))) - \mathcal{U}_\epsilon(\mathcal{E}_\epsilon(\tilde u_\epsilon(t)))| > \delta\},$ we must have $|\mathcal{E}_\epsilon(\tilde u_\epsilon(t+h)) - \mathcal{E}_\epsilon(\tilde{u}_\epsilon(t))| \geq \sigma_\delta$ (this is the contrapositive of \eqref{eq:equicts}). This implies from \eqref{eq:5d} that
\begin{align*}
\int_0^{T-h}\int_{\Omega_0}|\tilde{u}_\epsilon(t+h) - \tilde{u}_\epsilon(t)|\chi_{\{|\tilde{u}_\epsilon(t+h) - \tilde u_\epsilon(t)| > \delta\}} &\leq \frac{C_3\sqrt{h}}{\sigma_\delta}.
\end{align*}
Writing $\text{Id} = \chi_{\{|\tilde{u}_\epsilon(t+h) - \tilde u_\epsilon(t)| > \delta\}}+\chi_{\{|\tilde{u}_\epsilon(t+h) - \tilde u_\epsilon(t)| \leq \delta\}}$, notice that 
\begin{align*}
\int_0^{T-h}\int_{\Omega_0}|\tilde{u}_\epsilon(t+h) - \tilde{u}_\epsilon(t)| &\leq \int_0^{T-h}\int_{\Omega_0}|\tilde{u}_\epsilon(t+h) - \tilde{u}_\epsilon(t)|\chi_{\{|\tilde{u}_\epsilon(t+h) - \tilde u_\epsilon(t)| > \delta\}} + \delta|\Omega_0|(T-h)\\
&\leq  \frac{C_3\sqrt{h}}{\sigma_\delta} + \delta|\Omega_0|T.
\end{align*}
Taking the limit as $h \to 0$, using the arbitrariness of $\delta > 0$ and the fact that the right hand side of the above does not depend on $\epsilon$ gives us the result.
\end{proof}
\subsection{Existence of bounded weak solutions}
With all the uniform estimates acquired, we can extract (weakly) convergent subsequences. In fact, we find (we have not relabelled subsequences)
\begin{equation}\label{eq:listOfConvergences}
\begin{aligned}
u_\epsilon &\to u &&\text{in $L^p_{L^q}$ for any $p$, $q \in [1,\infty)$}\\
\sgrad u_\epsilon &\weaklyto \sgrad u &&\text{in $L^2_{L^2}$}\\
\mathcal E_{\epsilon}(u_{\epsilon}) &\weaklyto \chi &&\text{in $L^2_{L^2}$}%\\
%\md(\mathcal E_{\epsilon}(u_{\epsilon})) &\weaklyto \dot \chi &&\text{in $L^2_{H^{-1}}$}%\\
%\text{a.a. $t \in [0,T]$}, u_\epsilon(t) &\to u(t) &&\text{a.e. in $\Omega(t).$}
\end{aligned}
\end{equation}
where only the first strong convergence listed requires an explanation. Indeed, the point is to apply \cite[Theorem 5]{simon} with $H^1(\Omega_0) \compact L^1(\Omega_0) \subset L^1(\Omega_0),$  %Since $u_\epsilon$ is bounded in $L^2_{H^1}$, $\tilde u_\epsilon$ is bounded in $L^1(0,T;H^1(\Omega_0))$. The estimate on the time differences of $\tilde u_{\epsilon}$ above shows that the second condition of the theorem is satisfied. Thus there is 
which gives us a subsequence $\tilde u_{\epsilon_j} \to \tilde \rho$ strongly in $L^1(0,T;L^1(\Omega_0))$. It follows that $u_{\epsilon_j} \to \rho$ in $L^1_{L^1}$, whence for a.a. $t$, $u_{\epsilon_{j_k}}(t) \to \rho(t)$ a.e. in $\Omega(t)$. We also know that for a.a. $t$, $|u_{\epsilon_{j_k}}(t)| \leq C$ a.e. in $\Omega(t)$ by Lemma \ref{lem:energyEstimateLinfty}, and so for a.a. $t$, the limit satisfies $|\rho(t)| \leq C$ a.e. in $\Omega(t)$ too. By Theorem \ref{thm:DCT}, $u_{\epsilon_{j_k}} \to \rho$ in $L^p_{L^q}$ for all $p$, $q \in [1,\infty)$. Since $u_{\epsilon_{j_k}} \weaklyto u$ (subsequences have the same weak limit), it must be the case that $\rho = u$.% by the coincidence of the weak limit and the strong limit.

% OLD
%So we know that $u_{\epsilon'} \to \rho$ in $L^1_{L^1}$, whence for a.a. $t$, $u_{\epsilon''}(t) \to \rho(t)$ a.e. in $\Omega(t)$. Since $u_\epsilon \weaklyto u$ in $L^2_{L^2}$, we also have $u_{\epsilon''} \weaklyto u$ too. We want to show that $\rho=u.$

%We have $$\int_0^T \int_{\Omega(t)}u_{\epsilon''}(x,t)\eta(x,t) \to \int_0^T \int_{\Omega(t)}u(x,t)\eta(x,t)$$
%by weak convergence for every $\eta \in L^2_{L^2}$.

%Since $u_{\epsilon''}$ is bounded in $L^{\infty}_{L^\infty}$ uniformly, we have for a.a. $t$, $|u_{\epsilon''}(t)\eta(t)| \leq C|\eta(t)|$, and the right hand side is in $L^2_{L^2}$, and a.e. $t$, $u_{\epsilon''}(t) \to \rho(t)$ a.e. in $\Omega(t)$, so by DCT (Theorem \ref{thm:DCT}), the weak convergence (take the inner product in $L^2_{L^2}$ with $1 \in L^2_{L^2}$) $$\int_0^T \int_{\Omega(t)}u_{\epsilon''}(x,t)\eta(x,t) \to \int_0^T \int_{\Omega(t)}\rho(x,t)\eta(x,t)$$
%By uniqueness of weak limits, we must have $\rho = u$.
%So, we have in particular the convergences for a subsequence which we have relabelled:
%\begin{equation}
%\begin{aligned}
%u_\epsilon &\weaklyto u &&\text{in $L^2_{H^1}$}\\
%\mathcal E_{\epsilon}(u_{\epsilon}) &\weaklyto \chi &&\text{in $L^2_{L^2}$}\\
%\text{a.a. $t \in [0,T]$}, u_\epsilon(t) &\to u(t) &&\text{a.e. in $\Omega(t).$}
%\end{aligned}
%\end{equation}
%Let us now conclude the existence.
\begin{proof}[Proof of Theorem \ref{thm:existenceLinfty}]
%Testing \eqref{eq:2} with $\eta \in L^2_{H^1}$ such that $eta(T)  0$, integrating by parts and passing to the limit yields
%
%Passing to the limit in the weak form of \eqref{eq:2} using \eqref{eq:listOfConvergences} gives for all $\eta \in L^2_{H^1}$
%\begin{align*}
%\int_0^T \langle \md(\mathcal{E}_\epsilon(u_\epsilon(t))), \eta(t) \rangle_{\Vmt, \Vt} + \int_0^T \int_{\Omega(t)}\sgrad u_\epsilon(t) \sgrad \eta(t) + \int_0^T \int_{\Omega(t)}\mathcal{E}_\epsilon(u_\epsilon(t))\sgrad \cdot \mathbf w= \int_0^T \int_{\Omega(t)}f(t)\eta(t),
%\end{align*}
%in which we can pass to the limit using \eqref{eq:listOfConvergences}:
%\begin{align*}
%\int_0^T \langle \dot \chi(t), \eta(t) \rangle_{\Vmt, \Vt} + \int_0^T \int_{\Omega(t)}\sgrad u(t) \sgrad \eta(t) + \int_0^T \int_{\Omega(t)}\chi(t)\eta(t) \sgrad \cdot \mathbf w= \int_0^T \int_{\Omega(t)}f(t)\eta(t)
%\end{align*}
%and then we integrate by parts, supposing $\eta \in W(H^1, L^2)$ with $\eta(T)=0$:
%\begin{align*}
%-\int_0^T \int_{\Omega(t)}\dot \eta(t) \chi(t)+ \int_0^T \int_{\Omega(t)}\sgrad u(t) \sgrad \eta(t) = \int_0^T \int_{\Omega(t)}f(t)\eta(t) + \int_{\Omega_0}\chi(0) \eta(0).
%\end{align*}
%On the other hand, 
In \eqref{eq:2}, we can test with a function $\eta \in W(H^1, L^2)$ with $\eta(T) = 0$, integrate by parts and then pass to the limit to obtain
%\begin{align*}
%-\int_0^T \int_{\Omega(t)}\dot \eta(t)\mathcal{E}_\epsilon(u_\epsilon(t)) + \int_0^T \int_{\Omega(t)}\sgrad u_\epsilon(t) \sgrad \eta(t) = \int_0^T \int_{\Omega(t)}f(t)\eta(t) + \int_{\Omega_0}e_{0}\eta(0).
%\end{align*}
%We can pass to the limit in this equation:
\begin{align*}
-\int_0^T \int_{\Omega(t)}\dot \eta(t)\chi(t) + \int_0^T \int_{\Omega(t)}\sgrad u(t) \sgrad \eta(t) = \int_0^T \int_{\Omega(t)}f(t)\eta(t)  + \int_{\Omega_0}e_0\eta(0)
\end{align*}
%{and comparison with the above shows that $\chi(0) = e_0$,}
and it remains to be seen that $\chi \in \mathcal{E}(u)$ or equivalently $u = \mathcal{U}(\chi).$ By monotonicity of $\mathcal{E}_\epsilon$, we have for any $w \in L^2_{L^2}$
\begin{equation*}%\label{eq:monotonicityFormula}
\int_0^T\int_{\Omega(t)} (\mathcal{E}_\epsilon(u_\epsilon) - w)(u_\epsilon - \mathcal{U}_\epsilon(w)) \geq 0.
\end{equation*}
Because $\mathcal{U}_\epsilon \to \mathcal{U}$ uniformly, for a.a. $t$, $\mathcal{U}_\epsilon(w(t)) \to \mathcal{U}(w(t))$ a.e. in $\Omega(t)$, and $|\mathcal{U}_\epsilon (w)| \leq |w|$, and the dominated convergence theorem shows that $\mathcal{U}_\epsilon(w) \to \mathcal{U}(w)$ in $L^2_{L^2}$. Using this and \eqref{eq:listOfConvergences}, we can easily pass to the limit in this inequality and obtain
%To pass to the limit in this inequality, let us expand the brackets and consider each term in turn. We see that
%\begin{align*}
%\int_0^T\int_{\Omega(t)} w\mathcal{U}_\epsilon(w) &\to %\int_0^T\int_{\Omega(t)} w\mathcal{U}(w)\\
%\intertext{ We also trivially have}
%\int_0^T\int_{\Omega(t)} wu_\epsilon &\to \int_0^T\int_{\Omega(t)} %wu.
%\end{align*}
%As above, we have $\mathcal U_\epsilon(w) \to \mathcal U(w)$ in $L^2_{L^2}$, and because the inner product of a weakly convergent sequence with a strongly convergent sequence converges strongly, we have
%\[\int_0^T\int_{\Omega(t)}\mathcal{E}_\epsilon(u_\epsilon)\mathcal{U}_\epsilon( w) \to \int_0^T\int_{\Omega(t)}\chi \mathcal{U}(w).\]
%The same argument shows that
%\[\int_0^T\int_{\Omega(t)} \mathcal{E}_\epsilon(u_\epsilon)u_\epsilon \to \int_0^T\int_{\Omega(t)} \chi u.\]
%In other words, passing to the limit in \eqref{eq:monotonicityFormula}, we obtain the inequality
\[\int_0^T\int_{\Omega(t)}(\chi - w)(u- \mathcal{U}w) \geq 0 \quad \text{for all $w \in L^2_{L^2}$}.\]
%With Lemma \ref{lem:hemicontinuity} and 
By Minty's trick we find $u=\mathcal{U}(\chi)$. To see why $\chi \in L^\infty_{L^\infty}$, we have from the estimate in Lemma \ref{lem:energyEstimateLinfty} that 
for a.a. $t \in [0,T]$, $\norm{\mathcal{E}_\epsilon(\tilde u_\epsilon(t))}{L^\infty(\Omega_0)} \leq C,$ giving $\mathcal{E}_\epsilon(\tilde u_\epsilon(t)) \stackrel{*}\weaklyto \tilde \zeta(t)$ in $L^\infty(\Omega(t))$ and (by weak-* lower semicontinuity) $\lVert {\tilde \zeta(t)}\rVert_{L^\infty(\Omega(t))} \leq C$ for a.a. $t$, and we just need to identify $\tilde \zeta \in \mathcal{E}(\tilde u)$. It follows from \eqref{eq:listOfConvergences} that $\mathcal E_\epsilon(u_\epsilon) \to \chi$ in $L^2_{H^{-1}}$ by Lions--Aubin, and so for a.e. $t$ and for a subsequence (not relabelled), $\mathcal{E}_\epsilon(u_\epsilon(t)) \to \chi(t)$ in $H^{-1}(\Omega(t))$. This allows us to conclude that $\chi = \zeta$ (the weak-* convergence of $\mathcal{E}_\epsilon(\tilde u_\epsilon(t))$ to $\tilde \zeta(t)$ also gives weak convergence in any $L^p(\Omega(t))$ to the same limit).
\end{proof}
%\begin{remark}
%The result of Lemma \ref{lem:hemicontinuity} is a definition of hemicontinuity of $\mathcal U$; see \cite[p.~472]{zeidleriib}.
%\end{remark}
\subsection{Continuous dependence and uniqueness of bounded weak solutions}
{The next lemma allows us to drop the requirement for our test functions to vanish at time $T$.}
\begin{lem}\label{lem:dropT}%If $(u,e) \in L^1_{L^1}\times L^1_{L^1}$ satisfies
%\begin{equation}\label{eq:weakFormCtsDep}
%\begin{aligned}
%&\text{for all $\eta \in W^1(H^1, L^2)$ with $\eta(T) = 0$:}\\
%&-\int_0^T \int_{\Omega(t)}\dot \eta(t)e(t)+ \int_0^T \int_{\Omega(t)}\sgrad u \sgrad \eta(t) = \int_0^T \int_{\Omega(t)}f(t)\eta(t) + (e_0,\eta(0))_{L^2(\Omega_0)},
%\end{aligned}
%\end{equation}
If $(u,e)$ is a bounded weak solution (satisfying \eqref{eq:pdeWeakForm}), then $(u,e)$ also satisfies
\begin{equation*}
\begin{aligned}
%\text{for all $\eta \in W^1(H^1, L^2)$:}\qquad 
\int_{\Omega(T)}e(T)\eta(T) -\int_0^T \int_{\Omega(t)}\dot \eta(t)e(t)+ \int_0^T \int_{\Omega(t)}\sgrad u(t) \sgrad \eta(t)  = \int_0^T \int_{\Omega(t)}f(t)\eta(t) + \int_{\Omega_0}e_0\eta(0)
\end{aligned}
\end{equation*}
for all $\eta \in W(H^1, L^2)$.
\end{lem}
\begin{proof}
To see this, for $s \in (0,T]$, consider the function $\chi_{\epsilon,s}(t) = \min\left(1, \epsilon^{-1}(s-t)^+\right)$ which has a weak derivative $\chi_{\epsilon,s}'(t) = -\epsilon^{-1}\chi_{(s-\epsilon, s)}(t)$.
%\begin{equation*}
%\chi_{\epsilon,s}'(t) = \begin{cases}
%0 &: t \in (0,s-\epsilon) \\
%-\frac{1}{\epsilon} &: t \in (s-\epsilon, s)\\
%0 &: t \in (s,T) .
%\end{cases}.
%\end{equation*}
Take the test function in \eqref{eq:pdeWeakForm} to be $\chi_{\epsilon,T}\eta$ where $\eta \in W(H^1, L^2)$,
%\[-\int_0^T \int_{\Omega(t)}(\dot \chi_{\epsilon, T}(t)\eta(t) + \chi_{\epsilon, T}(t)\dot \eta(t)) e(t)+ \int_0^T \int_{\Omega(t)}\chi_{\epsilon, T}(t)\sgrad u \sgrad\eta(t) = \int_0^T\int_{\Omega(t)}\chi_{\epsilon, T}(t) f(t)\eta(t) + \int_{\Omega_0}\chi_{\epsilon, T}(0) e_0\eta(0)\]
%and this becomes
%\[\frac{1}{\epsilon}\int_{T-\epsilon}^T\int_{\Omega(t)} \eta (t)e(t)-\int_0^T\int_{\Omega(t)} \chi_{\epsilon, T}(t)\dot \eta (t) e(t)= \int_0^T\int_{\Omega(t)}\chi_{\epsilon, T}(t) \left( f(t)\eta (t)-\sgrad u(t) \sgrad \eta (t) \right) + \int_{\Omega_0}\chi_{\epsilon, T}(0) e_0\eta(0).\]
send $\epsilon \to 0$ and use the Lebesgue differentiation theorem. %on the left hand side 
%to find the desired equality.
%\[\int_{\Omega(T)}\varphi(T)e(T)-\int_0^T \int_{\Omega(t)}\dot \varphi(t)e(t)= \int_0^T\int_{\Omega(t)} \left(-\sgrad u(t)\sgrad \varphi(t) + \int_{\Omega(t)}f(t)\varphi(t)\right) + \int_{\Omega_0}e_0\eta(0).\]
\end{proof}
{We can finally prove Theorem \ref{thm:uniquenessAndCtsDependenceLinfty}. 
\begin{proof}[Proof of Theorem \ref{thm:uniquenessAndCtsDependenceLinfty}]We can prove the continuous dependence like in \cite[Chapter V, \S 9]{lady}.
%Recall that we want to show the continuous dependence result. %First, we prove that
%\begin{align*}
%\norm{e_1(t)-e_2(t)}{L^1(\Omega(t))} &\leq \int_0^t \norm{f_1(\tau)-f_2(\tau)}{L^1(\Omega(\tau))} + \norm{e_0^1-e_0^2}{L^1(\Omega_0)}
%\end{align*}
%holds.
 As explained in Lemma \ref{lem:dropT}, we drop the requirement $\eta(T) = 0$ in our test functions and we now suppose that $\slap \eta \in L^2_{L^2}$. Suppose for $i=1, 2$ that $(u_i, e_i)$ is the solution to the Stefan problem with data $(f_i, u_0^i)$, so
%\[\int_{\Omega(t)}e_i(t)\eta(t)-\int_0^t \langle \dot \eta(\tau), e_i(\tau)\rangle  + \int_0^t \int_{\Omega(\tau)}\sgrad u_i(\tau) \sgrad \eta(\tau) = \int_0^t \langle f_i(\tau), \eta(\tau) \rangle + (e_0^i, \eta(0))_{L^2(\Omega_0)}\]
%Taking the difference and if we assume $\eta(t) \in L^2(\Omega(t))$, we can integrate by parts:
\begin{align}
\nonumber \int_{\Omega(t)}(e_i(t)-e_2(t))\eta(t)-&\int_0^t \int_{\Omega(\tau)}\dot \eta(\tau)(e_1(\tau)-e_2(\tau))  - \int_0^t \int_{\Omega(\tau)}(u_1(\tau)-u_2(\tau)) \slap \eta(\tau)\\
&= \int_0^t \int_{\Omega(\tau)}(f_1(\tau)-f_2(\tau))\eta(\tau) + \int_{\Omega_0}(e_0^1-e_0^2)\eta(0).\label{eq:9}
\end{align}
Define $a={(u_1-u_2)}/({e_1-e_2})$ when $e_1 \neq e_2$ and $a=0$ otherwise, 
%\[a = \begin{cases}
%\frac{u_1-u_2}{e_1-e_2} &: e_1 \neq e_2\\
%0 &: e_1 = e_2
%\end{cases}
%\]
and note that $0 \leq a(x,t) \leq 1$. %The above becomes
%\begin{equation}
%\int_{\Omega(t)}(e_i(t)-e_2(t))\eta(t)-\int_0^t \int_{\Omega(\tau)}(e_1(\tau)-e_2(\tau))(\dot %\eta(\tau)  + a(x,\tau) \slap \eta(\tau)) = \int_0^t (f_1(\tau)-f_2(\tau))\eta(\tau) + \int_{\Omega_0}(e_0^1-e_0^2)\eta(0).
%\end{equation}
Let $\eta_\epsilon$ solve in $\cup_{\tau \in (0,t)}\{\tau\}\times \Omega(\tau)$ the equation
\begin{equation}\label{eq:12}
\begin{aligned}
\md_{\tau} \eta_\epsilon(\tau) + (a_\epsilon(x,\tau)+\epsilon)\slap \eta_\epsilon(\tau) &= 0&&\\
\eta_\epsilon(t) &= \xi &&\text{on $\Omega_0$}
\end{aligned}
\end{equation}
with $\xi \in C^1(\Omega_0)$ and where $a_\epsilon$ satisfies $\phi_{-(\cdot)}a_\epsilon \in C^2([0,T]\times\Omega_0)$ and $0 \leq a_\epsilon \leq 1$ a.e. and $\norm{a_\epsilon - a}{L^2(Q)} \leq \epsilon$. This is well-posed by Lemma \ref{lem:auxiliaryPDE}. Equation \eqref{eq:9} can be written in terms of $a_\epsilon$, and if we choose $\eta = \eta_\epsilon$ and use \eqref{eq:12}, we find
%\begin{align*}
%\int_{\Omega(t)}&(e_1(t)- e_2(t))\eta(t)-\int_0^t \int_{\Omega(\tau)}(e_1(\tau)-e_2(\tau))(\dot \eta(\tau)  + (a_\epsilon(x,\tau) +\epsilon)\slap \eta(\tau)-(a_\epsilon(x,\tau) - a(x,\tau) +\epsilon)\slap \eta(\tau))\\ 
%&= \int_0^t \int_{\Omega(\tau)} (f_1(\tau)-f_2(\tau)) \eta(\tau) \rangle + \int_{\Omega_0}(e_0^1-e_0^2)\eta(0)
%\end{align*}
%and if we choose $\eta = \eta_\epsilon$, using \eqref{eq:12},
%\begin{align*}
%\int_{\Omega(t)}(e_1(t)- e_2(t))\xi -\int_0^t \int_{\Omega(\tau)}(e_1(\tau)-e_2(\tau))( a(x,\tau) -a_\epsilon(x,\tau) -\epsilon)\slap \eta_\epsilon(\tau) &= \int_0^t \int_{\Omega(\tau)}(f_1(\tau)-f_2(\tau))\eta_\epsilon(\tau)\\
%&\quad + \int_{\Omega_0}(e_0^1-e_0^2)\eta_\epsilon(0).
%\end{align*}
%From this we obtain
\begin{align}
\nonumber \int_{\Omega(t)}(e_1(t)- e_2(t))\xi %&\leq \int_0^t \int_{\Omega(\tau)}|e_1(\tau)-e_2(\tau)|(|a(x,\tau) -a_\epsilon(x,\tau)| + \epsilon)|\slap \eta_\epsilon(\tau)| + \int_0^t \norm{f_1(\tau)-f_2(\tau)}{L^1(\Omega(\tau))}\norm{\eta_\epsilon(\tau)}{L^\infty(\Omega(\tau))}\\
&\leq \norm{e_1-e_2}{L^\infty_{L^\infty}}\int_0^t \int_{\Omega(\tau)}(|a(x,\tau) -a_\epsilon(x,\tau)| + \epsilon)|\slap \eta_\epsilon(\tau)| \\
&\quad+ \norm{\xi}{L^\infty(\Omega_0)}\int_0^t \norm{f_1(\tau)-f_2(\tau)}{L^1(\Omega(\tau))}+ \norm{\xi}{L^\infty(\Omega_0)}\int_{\Omega_0}|e_0^1-e_0^2|\label{eq:16}
\end{align}
using the $L^\infty$ bound from Lemma \ref{lem:auxiliaryPDE}. %We now use the results of Lemma \ref{lem:auxiliaryPDE}. 
We can estimate the first integral on the right hand side:
\begin{align*}
\int_0^t \int_{\Omega(\tau)}|a(x,\tau) -a_\epsilon(x,\tau)| |\slap \eta_\epsilon(\tau)| %&\leq \left(\int_0^t \int_{\Omega(\tau)}\frac{|a(x,\tau) -a_\epsilon(x,\tau)|^2}{\epsilon}\right)^{\frac 12}\left(\int_0^t \int_{\Omega(\tau)}\epsilon|\slap \eta_\epsilon(\tau)|^2\right)^{\frac 12}\\
&\leq \sqrt{\epsilon}\norm{a-a_\epsilon}{L^2_{L^2}}\sqrt{(2+\epsilon)(1+{e^{2C_{\mathbf w}(2+\epsilon)t}})}\norm{\sgrad \xi}{L^2(\Omega_0)}
\end{align*}
and
\begin{align*}
\int_0^t \int_{\Omega(\tau)}|\epsilon\slap\eta_\epsilon| %&\leq \left(\int_0^t\int_{\Omega(\tau)}\epsilon\right)^{\frac 12}\left(\int_0^t\int_{\Omega(\tau)}\epsilon(\slap \eta_\epsilon)^2\right)^{\frac 12}\\
&\leq \sqrt{t|\Omega|\epsilon (2+\epsilon)(1+{e^{2C_{\mathbf w}(2+\epsilon)t}})}\left(\int_{\Omega_0}|\sgrad \xi|^2\right)^{\frac 12} %\to 0
\end{align*}
by the results in Lemma \ref{lem:auxiliaryPDE}. Sending $\epsilon \to 0$ in \eqref{eq:16} gives us (recalling $\xi \leq 1$),
\begin{align*}
\int_{\Omega(t)}(e_1(t)- e_2(t))\xi &\leq \int_0^t \norm{f_1(\tau)-f_2(\tau)}{L^1(\Omega(\tau))} + \lVert{e_0^1-e_0^2}\rVert_{L^1(\Omega_0)}.
\end{align*}
Now pick $\xi = \xi_n$ where $\xi_n(x) \to \operatorname{sign}(e_1(t,x) - e_2(t,x)) \in L^2(\Omega(t))$ a.e. in $\Omega(t)$.% (recall that $t$ here is fixed).% to find the claimed result.
%\[\norm{e_1(t)- e_2(t)}{L^1(\Omega(t))} \leq \int_0^t \norm{f_1(\tau)-f_2(\tau)}{L^1(\Omega(\tau))} + \norm{e_0^1-e_0^2}{L^1(\Omega_0)}.\]
\end{proof}
%\begin{remark}
%In the proof above, we could for example pick the regularisation
%\[\xi_n(x) = \frac{e_1(t,x) - e_2(t,x)}{\sqrt{(e_1(t,x)-e_2(t,x))^2 + \frac{1}{n}}}.\]
%Then the integrand is bounded above by $\frac{(e_1(t,x) - e_2(t,x))^2}{\sqrt{(e_1(t,x)-e_2(t,x))^2}} = |e_1(t,x)-e_2(t,x)|$, which is integrable, hence by dominated convergence we have $\int_{\Omega(t)}(e_1(t)-e_2(t))\xi_n \to \int_{\Omega(t)}|e_1(t) - e_2(t)|$ as $n \to \infty$, and we obtain the desired result.
%\end{remark}
\subsection{Well-posedness of weak solutions}
%So far, we have proved well-posedness only for problems for which the initial data and the right hand side data are bounded.
\begin{proof}[Proof of Theorem \ref{thm:wellPosednessL1}]
Suppose $(e_0, f) \in L^1(\Omega_0) \times L^1_{L^1}$ are data and consider functions $e_{0n} \in L^\infty(\Omega_0)$ and $f_n \in L^\infty_{L^\infty}$ satisfying
\begin{equation*}
\begin{aligned}
(f_n, e_{0n}) &\to (f, e_0) &&\text{in $L^1_{L^1}\times L^1(\Omega_0)$}.
%f_n &\to f&&\text{in $L^1_{L^1}$}\\
\end{aligned}
\end{equation*}
The existence of $f_n$ holds because by density, there exist $\tilde f_n \in C^0([0,T]\times \Omega_0)$ such that $\tilde f_n \to \tilde f$ in $L^1((0,T)\times\Omega_0) \equiv L^1(0,T;L^1(\Omega_0))$. Denote by $(u_n, e_n)$ the respective (bounded weak) solutions to the Stefan problem with the data $(e_{0n}, f_{n})$. By virtue of these solutions satisfying
%\[\int_0^T \langle \dot e_n(t),\eta(t) \rangle_{\Vmt, \Vt} + \int_0^T \int_{\Omega(t)}\sgrad u_n(t) \sgrad \eta(t) + \int_0^T \int_{\Omega(t)}e_n(t)\eta(t)\sgrad \cdot \mathbf w = \int_0^T \int_{\Omega(t)}f_n(t)\eta(t)\]
the continuous dependence result,
%\begin{align*}
%\norm{e_n(t)- e_m(t)}{L^1(\Omega(t))} &\leq \int_0^t \norm{f_n(\tau)-f_m(\tau)}{L^1(\Omega(\tau))} + \norm{e_{0n}-e_{0m}}{L^1(\Omega_0)}\\
%&\leq \norm{f_n-f_m}{L^1_{L^1}} + \norm{e_{0n}-e_{0m}}%{L^1(\Omega_0)}
%\end{align*}
%which we can integrate over $[0,T]$:
%\begin{align*}
%\norm{e_n- e_m}{L^1_{L^1}} &\leq T\left(\norm{f_n-f_m}{L^1_{L^1}} + \norm{e_{0n}-e_{0m}}{L^1(\Omega_0)}\right).
%\end{align*}
it follows that $\{e_n\}_{n}$ is a Cauchy sequence in $L^1_{L^1}$ and thus $e_n \to \chi$ in $L^1_{L^1}$ for some $\chi$. Recall that $|u_n| = |\mathcal{U}(e_n)| \leq |e_n|$ %so we can consider the Nemytskii map $N_{\mathcal{U}}\colon L^1_{L^1} \to L^1_{L^1}$ defined by $N_{\mathcal{U}}(g) = \mathcal{U}(g)$. Clearly ${\mathcal{U}}$ is continuous and hence Caretheodory and so $N_{\mathcal{U}}$ is continuous, which means that 
so by consideration of an appropriate Nemytskii map, we find $u_n = \mathcal{U}(e_n) \to \mathcal{U}(\chi)$. Now we can pass to the limit in 
\begin{align*}
-\int_0^T \int_{\Omega(t)}\dot \eta(t)e_n(t) - \int_0^T \int_{\Omega(t)}u_n(t) \slap \eta(t) &= \int_0^T \int_{\Omega(t)}f_n(t)\eta(t) + \int_{\Omega_0}e_{n0} \eta(0)
\end{align*}
and doing so gives
\begin{align*}
-\int_0^T \int_{\Omega(t)}\dot \eta(t)\chi(t) - \int_0^T \int_{\Omega(t)}\mathcal{U}(\chi(t)) \slap \eta(t) &= \int_0^T \int_{\Omega(t)}f(t)\eta(t) + \int_{\Omega_0}e_0\eta(0)
\end{align*}
and overall this shows that there exists a pair $(\chi, \mathcal{E}^{-1}(\chi)) \in L^1_{L^1} \times L^1_{L^1}$ which is a weak solution of the Stefan problem. For these integrals to make sense, we need $\eta \in W^1(L^\infty\cap H^2, L^\infty)$ with $\slap \eta \in L^\infty_{L^\infty}$.
%Recall that $u_n = \mathcal{E}^{-1}(e_n)$ and observe that $|u_n| \leq |e_n|$ and so $u_n$ is bounded in $L^1_{L^1}$. Thus $u_n \weaklyto u$ in $L^1_{L^1}$ for a subsequence (which we have relabelled) for some $u$. We pass to the limit in
%\begin{align*}
%-\int_0^T \int_{\Omega(t)}\dot \eta(t)e_n(t) - \int_0^T \int_{\Omega(t)}u_n(t) \slap \eta(t) &= \int_0^T \int_{\Omega(t)}f_n(t)\eta(t) + (e_{n0}, \eta(0))_{L^2(\Omega_0)}
%\end{align*}
%and doing so gives
%\begin{align*}
%-\int_0^T \int_{\Omega(t)}\dot \eta(t)\chi(t) - \int_0^T \int_{\Omega(t)}u(t) \slap \eta(t) &= \int_0^T \int_{\Omega(t)}f(t)\eta(t) + \int_{\Omega_0}e_0\eta(0).
%\end{align*}
%For this integral to make sense, we need $\eta \in W^1(L^\infty\cap H^2, L^\infty)$ with $\slap \eta \in L^\infty_{L^\infty}$.

%It remains to check that $\chi = e$, the solution of the Stefan problem with data $(u_0, f)$, or that $u = \mathcal{E}^{-1}(\chi)$. This follows because since $e_n \to \chi$ in $L^1_{L^1}$, $e_n(t) \to \chi(t)$ pointwise a.e. in $\Omega(t)$. By continuity, $u_n(t) \to  \mathcal{E}^{-1}(\chi(t))$ pointwise a.e. {This shows that $u=\mathcal{E}^{-1}(\chi)$ because the pointwise limit and the weak limit must coincide. Is there an easier way to do this?} 
Now suppose that $(u^1, e^1)$ and $(u^2, e^2)$ are two weak solutions of class $L^1$ to the Stefan problem with data $(f^1, e^1_0)$ and $(f^2, e^2_0)$ in $L^1_{L^1}\times L^1(\Omega_0)$ respectively. We know that there exist approximations $(f^1_n, e^1_{0n})$, $(f^2_n, e^2_{0n}) \in L^\infty_{L^\infty}\times L^\infty(\Omega_0)$ of the data satisfying
\begin{equation*}
\begin{aligned}
(f^1_n, e^1_{0n}) &\to (f^1, e^1_0) \quad \text{and} \quad (f^2_n, e^2_{0n}) \to (f^2, e^2_0)&&\text{in $L^1_{L^1} \times L^1(\Omega_0)$}.
\end{aligned}
\end{equation*}
These approximate data give rise to the approximate solutions $e^1_n$ and $e^2_n$ both of which are elements of $L^{\infty}_{L^\infty}$. It follows from above that $e^1_n \to e^1$ and $e^2_n \to e^2$ in $L^1_{L^1}$. Now consider the continuous dependence result that $e^1_n$ and $e^2_n$ satisfy:
\begin{equation}\label{eq:16a}
\lVert{e^1_n- e^2_n}\rVert_{L^1_{L^1}} \leq T\left(\lVert{f^1_n-f^2_n}\rVert_{L^1_{L^1}} + \lVert{e^1_{0n}-e^2_{0n}}\rVert_{L^1(\Omega_0)}\right).
\end{equation}
Regarding the right hand side, by writing $e^1_{0n}- e^2_{0n} = e^1_{0n} - e^1_0 + e^1_0 - e^2_0+ e^2_0 - e^2_{0n},$
%\begin{align*}
%\norm{e^1_{0n}- e^2_{0n}}{L^1_{L^1}} &\leq \norm{e^1_{0n} - e^1_0}{L^1_{L^1}}+ \norm{e^1_0 - e^2_0}{L^1_{L^1}} + \norm{e^2_0 - e^2_{0n}}{L^1_{L^1}},
%\end{align*}
(and similarly for the $f^i_n$) and using triangle inequality, along with the fact that $e^1_n-e^2_n \to e^1-e^2$ in $L^1_{L^1}$, we can take the limit in \eqref{eq:16a} as $n \to \infty$ and we are left with what we desired.
%\begin{equation*}
%\norm{e^1- e^2}{L^1_{L^1}} \leq T\left(\norm{f^1-f^2}{L^1_{L^1}} + \norm{e^1_{0}-e^2_{0}}{L^1(\Omega_0)}\right).
%\end{equation*}
\end{proof}

\appendix
\section{Proofs}
\begin{proof}[Proof that $L^p_X$ is a Banach space when $X_0$ is separable]\label{item:A}
It is easy to verify that the expressions in \eqref{eq:ip} define norms if the integrals on the right hand sides are well-defined, which we now check. So let $u \in L^p_X$. Then $\tilde u := \phi_{-(\cdot)}u(\cdot) \in L^p(0,T;X_0)$. Define $F\colon [0,T] \times X_0 \to \mathbb{R}$ by $F(t,x) = \norm{\phi_t x}{X(t)}$. By assumption, $t \mapsto F(t,x)$ is measurable for all $x \in X_0$, and if $x_n \to x$ in $X_0$, then by the reverse triangle inequality,
\begin{align*}
|F(t,x_n) - F(t,x)| \leq \norm{\phi_t(x_n-x)}{X(t)} \leq C_X\norm{x_n -x}{X_0} \to 0,
\end{align*} 
so $x \mapsto F(t,x)$ is continuous. Thus $F$ is a Carath\'eodory function. Due to the condition $|F(t,x)| \leq C_X\norm{x}{X_0}$, by Remark 3.4.5 of \cite{gasinski}, the Nemytskii operator $N_F$ defined by $(N_F x)(t) := F(t,x(t))$ maps  $L^p(0,T;X_0) \to L^p(0,T)$, so that
\[\norm{N_F \tilde u}{L^p(0,T)}^p = \int_0^T \norm{u(t)}{X(t)}^p < \infty.\]
\end{proof}
\begin{proof}[Proof of Lemma \ref{lem:pushforwardIso}]\label{item:B}
First we show that if $u \in L^p(0,T;X_0)$, then $\phi_{(\cdot)} u(\cdot) \in L^p_X$.

Let $u \in L^p(0,T;X_0)$ be arbitrary. By density, there exists a sequence of simple functions $u_n \in L^p(0,T;X_0)$ with 
$$\norm{u_n - u}{L^p(0,T;X_0)} \to 0$$
and thus for almost every $t$,
$$\norm{u_n(t) - u(t)}{X_0} \to 0$$
for a subsequence, which we relabelled. We have that $\phi_t u_n(t) \to \phi_t u(t)$ in $X(t)$ by continuity; this implies
%$$\bigg|\norm{\phi_t u_n(t)}{X(t)} - \norm{\phi_t u(t)}{X(t)}\bigg| \leq\norm{\phi_t u_n(t) - \phi_t u(t)}{X(t)} \to 0$$
%so that 
\begin{equation}\label{eq:pwl}
\norm{\phi_t u_n(t)}{X(t)} \to \norm{\phi_t u(t)}{X(t)}\qquad\text{pointwise a.e.} 
\end{equation}
Write $u_n(t) = \sum_{i=1}^{M_n}u_{n,i}\mathbf{1}_{B_i}(t)$ where the $u_{n,i} \in X_0$ and the $B_i$ are measurable, disjoint and partition $[0,T].$  Then
$$\phi_t u_n(t) = \sum_{i=1}^{M_n}\phi_t(u_{n,i})\mathbf{1}_{B_i}(t) \in X(t).$$
Taking norms and exponentiating, we get
$$\norm{\phi_t u_n(t)}{X(t)}^p = \sum_{i=1}^{M_n}\norm{\phi_tu_{n,i}}{X(t)}^p\mathbf{1}_{B_i}^p(t),$$
which is measurable (with respect to $t$) since, by assumption, the $\norm{\phi_tu_{n,i}}{X(t)}$ are continuous and a finite sum of measurable functions is measurable. Thus, by \eqref{eq:pwl}, $\norm{\phi_t u(t)}{X(t)}$, is measurable. Finally, 
$$\int_0^T \norm{\phi_t u(t)}{X(t)}^p \leq \int_0^T C_X^p \norm{ u(t)}{X_0}^p = C_X^p\norm{u}{L^p(0,T;X_0)}^p,$$
so $\phi_{(\cdot)} u(\cdot) \in L^p_X$.

So there is a map from $L^p(0,T;X_0)$ to $L^p_X$ and vice-versa from the definition of $L^p_X$. The isomorphism between the spaces is $T\colon L^p(0,T;X_0) \to L^p_X$ where 
\begin{align*}
Tu = \phi_{(\cdot)}u(\cdot),\quad\text{and}\quad
T^{-1}v = \phi_{-(\cdot)}v(\cdot).
\end{align*}
It is easy to check that $T$ is linear and bijective. The equivalence of norms follows by the bounds on $\phi_{-t}\colon X(t) \to X_0$ 
\[\frac{1}{C_X}\norm{u(t)}{X(t)} \leq \norm{\phi_{-t}u(t)}{X_0} \leq C_X\norm{u(t)}{X(t)}.\]
\end{proof}
\section*{Acknowledgments}
A.A. was supported by the Engineering and Physical Sciences Research Council (EPSRC) Grant EP/H023364/1 within the MASDOC Centre for Doctoral Training. This work was initiated at the Isaac Newton Institute in Cambridge during the \emph{Free Boundary Problems and Related Topics} programme (January -- July 2014). {The authors are grateful to the referees for their useful feedback and encouragement.}

%%%%%%%%%% Insert bibliography here %%%%%%%%%%%%%%

\end{document}